\documentclass[UTF8,fontset=windows,final, leqno]{siamltex}
\usepackage{ctex}
\usepackage{amsmath}
\usepackage{epsfig}
\usepackage{graphicx}
\usepackage{amssymb}
\usepackage{CJK}
\usepackage{color}
\usepackage[noadjust]{cite}
\usepackage{hyperref}
\usepackage{caption}
\usepackage{float}

\numberwithin{equation}{section}
\newtheorem{remark}{Remark}[section]

\allowdisplaybreaks[4]

\renewcommand\abstractname{Abstract}

\renewcommand\figurename{Figure}
\renewcommand\tablename{Table}

\renewcommand\refname{References}
\def\lam{{\lambda}}

\def\Ome{{\Omega}}

\def\nab{{\nabla}}
\def\vepsi{{\varepsilon}}
\def\p{{\partial}}
\def\reff#1{\eqref{#1}}
\def\norm#1#2{\Vert\,#1\,\Vert_{#2}}

\def\vepsi{\varepsilon}

\def\cT{{\mathcal T}}

\def\no{{\nonumber}}

\def\div{{\mbox{\rm div\,}}}

\def\p{{\partial}}

\def\nab{\nabla}
\def\Ome{\Omega}
\def\lam{\lambda}

\newcommand{\bRM}{\mathbf{RM}}
\newcommand{\br}{\mathbf{r}}

\def\ba{\mathbf{a}}
\def\bb{\mathbf{b}}

\def\bC{\mathbf{C}}
\def\bbf{\mathbf{f}}
\def\bu{\mathbf{u}}

\def\bv{\mathbf{v}}
\def\bw{\mathbf{w}}

\def\bg{\mathbf{g}}
\def\bn{\mathbf{n}}
\def\bH{\mathbf{H}}
\def\bV{\mathbf{V}}
\def\bL{\mathbf{L}}
\def\bP{\mathbf{P}}

\def\bV{\mathbf{V}}

\def\bX{\mathbf{X}}

\def\R{\mathbb{R}}

\def\bx{{\bf x}}
\CTEXoptions[today=old]

\begin{document}

%    \ifpdf
%    \DeclareGraphicsExtensions{.pdf, .jpg, .tif, .png}
%    \else
%    \DeclareGraphicsExtensions{, .jpg}
%    \fi

%\title{Multiphysics Finite Element Methods for a Poroelasticity Model}
\title{Error Estimates of a Fully Discrete Multiphysics Finite Element Method for a Nonlinear Poroelasticity Model\footnote{Last update: \today}}

\author{
Zhihao Ge\thanks{School of Mathematics and Statistics, Henan University, Kaifeng 475004, P.R. China ({\tt zhihaoge@henu.edu.cn}).
	The work of this author was supported by the National Natural Science Foundation of China under grant No. 11971150.}
\and
Wenlong He\thanks{School of Mathematics and Statistics, Henan University, Kaifeng 475004,P.R.China}
}

%%%
\maketitle

%\vspace{-1.4in}
%\slugger{sinum}{200x}{xx}{x}{xxx--xxx}
%\vspace{1.1in}

\setcounter{page}{1}

%\begin{PII}
%S00000000000
%\end{PII}

%\large

\begin{abstract}
In this paper, we propose a multiphysics finite element method for a nonlinear poroelasticity model. To better describe the processes of deformation and diffusion, we firstly reformulate the nonlinear fluid-solid coupling problem into a fluid-fluid coupling problem by a multiphysics approach. Then we design a fully discrete time-stepping scheme to use multiphysics finite element method with $P_2-P_1-P_1$ element pairs  for the space variables and backward Euler method for the time variable, and we adopt the Newton iterative method to deal with the nonlinear term. Also, we derive the discrete energy laws and the optimal convergence order error estimates without any assumption on the nonlinear stress-strain relation. Finally, we show some numerical examples to verify the rationality of theoretical analysis and there is no ``locking phenomenon".
\end{abstract}


\pagestyle{myheadings}
\thispagestyle{plain}
\markboth{ZHIHAO GE, WENLONG HE}{MULTIPHYSICS FINITE ELEMENT METHODS FOR NONLINEAR POROELASTICITY}

%%%%%%%%%%%%%%%%%%%%%%%%%%%%%%

\section{Introduction}\label{sec-1}

Poroelasticity model is a fluid-solid coupled system at poro scale, which is widely used in various fields such as geophysics, biomechanics, civil engineering, chemical engineering, materials science and so on, one can refer to \cite{2,3,6,7,8,10,biot,coussy04,de86}. There are many kinds of nonlinear poroelasticity, such as the permeability tensor $ K(\div \bu) $ (cf. \cite{20210921,20210922,20210923,20210924,20210925,20210927,20210928} and the references therein), the nonlinear constitutive stress-strain of solid (cf. \cite{1,11}) and so on.  In practical applications, many problems, such as cables, beams, shells, polymers and metal foams, require a nonlinear stress-strain relation, one can refer to \cite{201912092,2,coussy04,hamley07} and so on.  For linear poroelasticity, Showalter provides the analysis of well-posedness of weak solution to a linear poroelasticity model in \cite{20210819}. Phillips and Wheeler propose and analyze a continuous-in-time and a discrete-in-time mixed finite element method in \cite{pw07,pw07b} which simultaneously approximates the pressure and its gradient along with the displacement vector field, and the authors pointed that there exists ``locking phenomenon" by using continuous Galerkin finite method. Feng, Ge and Li in \cite{fglarxiv,fgl14} propose a multiphysics finite method for approximating linear poroelasticity model by reformulating the original model, the multiphysics finite element method is a effective approach to study the poroelasticity model and it overcomes the ``locking phenomenon". Based on the idea of \cite{fgl14}, Ge and He in \cite{201912096} prove the growth, coercivity and monotonicity of $\mathcal{N}(\nabla\bu)$ based on a multiphysics approach without any assumption on the nonlinear stress-strain relation for a nonlinear poroelasticity model with the constitutive relation $\tilde{\sigma}(\mathbf{u})=\mu\tilde{\varepsilon}(\mathbf{u})+\lambda tr(\tilde{\varepsilon}(\mathbf{u}))\mathbf{I}$, where $\tilde{\varepsilon}(\bu)=\dfrac{1}{2}(\nabla\bu+\nabla^{T}\boldsymbol{u}+2\nabla^{T}\bu\nabla\bu)$ is the deformed Green strain tensor. In this paper, we propose a fully discrete multiphysics finite element method for the nonlinear poroelasticity model(cf. \cite{201912096})  by using the $P_2-P_1-P_1$ element pairs for space variables and backward Euler method for time variable. 
%The most challenge is to deal with the nonlinear term in computation by Newton iterative method. 
Without any assumption on the nonlinear stress-strain relation, we derive the discrete energy estimates and apply Schauder's fixed point theorem to prove the existence and uniqueness of the numerical solution of the proposed numerical method. And we prove that the time-stepping method has the optimal convergence order. In the numerical tests, we show some numerical examples to verify the theoretical results and there is no ``locking phenomenon". To the best of our knowledge, it is the first time to propose a fully discrete multiphysics finite element method and derive the optimal order error estimate for the nonlinear poroelasticity model.  

The remainder of this paper is organized as follows. In Section \ref{sec-3.2}, we introduce the basic results of PDE model. In Section \ref{sec-3}, we propose and analyze the coupled and decoupled time stepping methods based on the multiphysics approach. In Section \ref{sec-3.3}, we prove that the time-stepping has the optimal convergence order. In Section \ref{sec-4}, we provide some numerical experiments to verify the theoretical results of the proposed approach and methods. Finally, we draw a conclusion to summary the main results of this paper.

\section{ Basic results of PDE model}\label{sec-3.2}
In this paper, we consider the following quasi-static poroelasticity model (for the linear and nonlinear poroelasticity model, one can refer to \cite{pw07,fgl14,fglarxiv} and \cite{201912096}, respectively):
\begin{alignat}{2}\label{8.71} 
	-\div\tilde{\sigma}(\bu) + \alpha \nab p &= \bbf
	&&\qquad \mbox{in } \Ome_T:=\Ome\times (0,T)\subset \mathbf{\R}^d\times (0,T),\\
	(c_0p+\alpha \div \bu)_t + \div \bv_f &=\phi &&\qquad \mbox{in } \Ome_T,\label{8.72}
\end{alignat}
where
\begin{align} 
	&\tilde{\sigma}(\mathbf{u})=\mu\tilde{\varepsilon}(\mathbf{u})+\lambda tr(\tilde{\varepsilon}(\mathbf{u}))\mathbf{I},~~~~~ \tilde{\varepsilon}(\bu)=\dfrac{1}{2}(\nabla\bu+\nabla^{T}\bu+2\nabla^{T}\bu\nabla\bu), \label{8.73}\\
	&\bv_f:= -\frac{K}{\mu_f} \bigl(\nab p -\rho_f \bg \bigr). \label{8.74}
\end{align}
Here $\tilde{\vepsi}(\bu)$ is known as the deformed Green strain tensor, $\bu$ denotes the displacement vector of the solid and $p$ denotes the pressure of the solvent. $\mathbf{I}$ denotes the $d\times d$ identity matrix. $\bbf$ is the body force. The permeability tensor $K=K(x)$ is assumed to be symmetric and uniformly positive definite in the sense that there exists positive constants $K_1$ and $K_2$ such that $K_1|\zeta|^2\leq K(x)\zeta\cdot \zeta \leq K_2 |\zeta|^2$ for a.e. $x\in\Omega$ and any $\zeta\in \mathbf{\R}^d$; the solvent viscosity $\mu_f$, Biot-Willis constant $\alpha$, and the constrained specific storage coefficient $c_0$. In addition, $\tilde{\sigma}(\bu)$ is called the (effective) stress tensor. $\bv_f$ is the volumetric solvent flux and (\ref{8.74}) is called the well-known Darcy's law. $ \lambda $ and $ \mu $ are
Lam\'e constants, $\widehat{\sigma}(\bu, p):=\tilde{\sigma}(\bu)-\alpha p \mathbf{I}$ is the total stress tensor. We assume that $\rho_f\not\equiv 0$, which is a realistic assumption.

To close the above system, the following set of boundary and initial conditions will be considered in this paper:
\begin{alignat}{2} \label{8.75}
	\widehat{\sigma}(\bu,p)\bn=\tilde{\sigma}(\bu)\bn-\alpha p \bn &= \bbf_1
	&&\qquad \mbox{on } \p\Ome_T:=\p\Ome\times (0,T),\\
	\bv_f\cdot\bn= -\frac{K}{\mu_f} \bigl(\nab p -\rho_f \bg \bigr)\cdot \bn
	&=\phi_1 &&\qquad \mbox{on } \p\Ome_T, \label{8.76} \\
	\bu=\bu_0,\qquad p&=p_0 &&\qquad \mbox{in } \Ome\times\{t=0\}. \label{8.77}
\end{alignat}
Introduce new variables
\[
q:=\div \bu,\quad \eta:=c_0p+\alpha q,\quad \xi:=\alpha p -\lam q.
\]
Denote 
\begin{eqnarray}
	\mathcal{N}(\nabla\bu)=\tilde{\sigma}(\bu)-\lambda \div\bu~ \mathbf{I},\label{eq210823-1}
\end{eqnarray}  
then we have
\begin{eqnarray}
	\mathcal{N}(\nabla\mathbf{u})=\mu\varepsilon(\mathbf{u})+\mu \nabla^{T}\mathbf{u}\nabla\mathbf{u}+\lambda\|\nabla\mathbf{u}\|_{F}^{2}\mathbf{I}.\label{eq210823-2}
\end{eqnarray}
Due to the fact of $(\nabla^{T}\mathbf{u}\nabla\mathbf{u},rot \mathbf{v})=0,~(\|\nabla\mathbf{u}\|_{F}^{2}{\bf I},rot \mathbf{v})=0$, so we have
\begin{equation*}
	(\mathcal{N}(\nabla\mathbf{u}),\nabla\mathbf{v})=(\mathcal{N}(\nabla\mathbf{u}),\varepsilon(\mathbf{v})),
\end{equation*}
where $\varepsilon(\bu)=\dfrac{1}{2}(\nabla^{T}\bu+\nabla\bu)$.

In some engineering literature, the Lam\'e constant
$\mu$ is also called the {\em shear modulus} and denoted by $G$, and
$B:=\lam +\frac23 G$ is called the {\em bulk modulus}. $\lam,~\mu$ and $B$
are computed from the {\em Young's modulus} $E$ and the {\em Poisson ratio}
$\nu$ by the following formulas:
\[
\lam=\frac{E\nu}{(1+\nu)(1-2\nu)},\qquad \mu=G=\frac{E}{2(1+\nu)}, \qquad
B=\frac{E}{3(1-2\nu)}.
\]

It is easy to check that
\begin{align}\label{2.19}
	p=\kappa_1 \xi + \kappa_2 \eta, \qquad q=\kappa_1 \eta-\kappa_3 \xi,
\end{align}
where $\kappa_1= \frac{\alpha}{\alpha^2+\lam c_0},
\kappa_2=\frac{\lam}{\alpha^2+\lam c_0},
\kappa_3=\frac{c_0}{\alpha^2+\lam c_0}$.

Then the problem \reff{8.71}-\reff{8.74} can be rewritten as
\begin{alignat}{2} \label{2.3}
	-\div\mathcal{N}(\nabla\bu) + \nab \xi &= \bbf &&\qquad \mbox{in } \Ome_T,\\
	\kappa_3\xi +\div \bu &=\kappa_1\eta &&\qquad \mbox{in } \Ome_T, \label{2.4}\\
	\eta_t - \frac{1}{\mu_f} \div[K (\nab (\kappa_1 \xi + \kappa_2 \eta)-\rho_f\bg)]&=\phi
	&&\qquad \mbox{in } \Ome_T. \label{2.5}
\end{alignat}
The boundary and initial conditions \reff{8.75}-\reff{8.77} can be rewritten as
\begin{alignat}{2} \label{2.9}
	\tilde{\sigma}(\bu)\bn-\alpha (\kappa_1 \xi + \kappa_2 \eta) \bn &= \bbf_1
	&&\qquad \mbox{on } \p\Ome_T:=\p\Ome\times (0,T),\\
	-\frac{K}{\mu_f} \bigl(\nab (\kappa_1 \xi + \kappa_2 \eta) -\rho_f \bg \bigr)\cdot \bn
	&=\phi_1 &&\qquad \mbox{on } \p\Ome_T, \label{2.10} \\
	\bu=\bu_0,\qquad p&=p_0 &&\qquad \mbox{in } \Ome\times\{t=0\}. \label{2.11}
\end{alignat}
In this paper, $\Omega \subset \R^d \,(d=1,2,3)$ denotes a bounded polygonal domain with the boundary
$\p\Ome$. The standard function space notation is adopted in this paper, their
precise definitions can be found in \cite{bs08,cia,temam}.
In particular, $(\cdot,\cdot)$ and $\langle \cdot,\cdot\rangle$
denote respectively the standard $L^2(\Ome)$ and $L^2(\p\Ome)$ inner products. For any Banach space $B$, we let $\mathbf{B}=[B]^d$, and denote  its dual space by $\mathbf{B}^\prime$. In particular, we use $(\cdot,\cdot)_{\small\rm dual}$ %and $\langle \cdot,\cdot \rangle_{\small\rm dual}$
to denote the dual product on $\bH^1(\Ome)' \times \bH^1(\Ome)$,
and $\norm{\cdot}{L^p(B)}$ is a shorthand notation for
$\norm{\cdot}{L^p((0,T);B)}$.

We also introduce the function spaces
\begin{align*}
	&L^2_0(\Omega):=\{q\in L^2(\Omega);\, (q,1)=0\}, \qquad \bX:= \bH^1(\Ome).
\end{align*}
Let $\bRM:=\{\br:=\ba+\bb \times x;\, \ba, \bb, x\in \R^d\}$ denote the space of infinitesimal rigid motions. It is well known (cf. \cite{brenner,gra,temam}) that $\bRM$ is the kernel of
the strain operator $\vepsi$, that is, $\br\in \bRM$ if and only if
$\vepsi(\br)=0$. Hence, we have
\begin{align}
	\vepsi(\br)=0,\quad \div \br=0 \qquad\forall \br\in \bRM. \label{2.2}
\end{align}

Let $L^2_\bot(\p\Ome)$ and $\bH^1_\bot(\Ome)$ denote respectively the subspaces of $L^2(\p\Ome)$ and $\bH^1(\Ome)$ which are orthogonal to $\bRM$, that is,
\begin{align*}
	&\bH^1_\bot(\Ome):=\{\bv\in \bH^1(\Ome);\, (\bv,\br)=0\,\,\forall \br\in \bRM\},
	\\
	&L^2_\bot(\p\Ome):=\{\bg\in L^2(\p\Ome);\,\langle \bg,\br\rangle=0\,\,
	\forall \br\in \bRM \}.
\end{align*}

Next, we introduce the definition of the weak solution to the problem \reff{8.71}-\reff{8.77} and \reff{2.3}-\reff{2.5} with
\reff{2.9}-\reff{2.11}  as follows.
\begin{definition}\label{weak1}
Let $\bu_0\in\bH^1(\Ome),~ \bbf\in\bL^2(\Omega),~
\bbf_1\in \bL^2(\p\Ome),~ p_0\in L^2(\Ome),~ \phi\in L^2(\Ome)$,
and $\phi_1\in  L^2(\p\Ome)$.  Assume $c_0>0$ and
$(\bbf,\bv)+\langle \bbf_1,~ \bv \rangle =0$ for any $\bv\in \mathbf{RM}$.
Given $T > 0$, a tuple $(\bu,p)$ with
\begin{alignat*}{2}
	&\bu\in L^\infty\bigl(0,T; \bH_\perp^1(\Ome)),
	&&\qquad p\in L^\infty(0,T; L^2(\Omega))\cap L^2 \bigl(0,T; H^1(\Omega)\bigr), \\
	&p_t, (\div\bu)_t \in L^2(0,T;H^{1}(\Ome)')
	&&\qquad %c_0^{\frac12} p\in L^\infty\bigl(0,T; L^2(\Ome)),
\end{alignat*}
is called a weak solution to the problem \reff{8.71}--\reff{8.77}, if there hold for almost every $t \in [0,T]$
\begin{alignat}{2}\label{2.32}
	&\bigl( \mathcal{N}(\nabla\bu), \vepsi(\bv) \bigr)
	+\lam\bigl(\div\bu, \div\bv \bigr)
	-\alpha \bigl( p, \div \bv \bigr)  && \\
	&\hskip 2in
	=(\bbf, \bv)+\langle \bbf_1,\bv\rangle
	&&\quad\forall \bv\in \bH^1(\Ome), \no \\
	&\bigl((c_0 p +\alpha\div\bu)_t, \varphi \bigr)_{\rm dual}
	+ \frac{1}{\mu_f} \bigl( K(\nab p-\rho_f\bg), \nab \varphi \bigr)
	\label{2.33} \\
	&\hskip 2in =\bigl(\phi,\varphi\bigr)
	+\langle \phi_1,\varphi \rangle
	&&\quad\forall \varphi \in H^1(\Ome), \no  \\
	&\bu(0) = \bu_0,\qquad p(0)=p_0.  && \label{2.34}
\end{alignat}
\end{definition}
\begin{definition}\label{weak2}
	Let $\bu_0\in \bH^1(\Ome),~ \bbf \in \bL^2(\Omega),~
	\bbf_1 \in \bL^2(\p\Ome),~ p_0\in L^2(\Ome),~ \phi\in L^2(\Ome)$,
	and $\phi_1\in L^2(\p\Ome)$.  Assume $c_0>0$ and
	$(\bbf,\bv)+\langle \bbf_1,~ \bv \rangle =0$ for any $\bv\in \mathbf{RM}$.
	Given $T > 0$, a $5$-tuple $(\bu,\xi,\eta,p,q)$ with
	\begin{alignat*}{2}
		&\bu\in L^\infty\bigl(0,T; \bH_\perp^1(\Ome)), &&\qquad
		\xi\in L^\infty \bigl(0,T; L^2(\Omega)\bigr), \\
		&\eta\in L^\infty\bigl(0,T; L^2(\Omega)\bigr)
		\cap H^1\bigl(0,T; H^{1}(\Omega)'\bigr),
		&&\qquad q\in L^\infty(0,T;L^2(\Ome)), \\
		&p\in L^\infty \bigl(0,T; L^2(\Omega)\bigr) \cap L^2 \bigl(0,T; H^1(\Omega)\bigr)  &&
	\end{alignat*}
	is called a weak solution to the problem \reff{2.3}-\reff{2.5},	if there hold for almost every $t \in [0,T]$
	\begin{alignat}{2}\label{2.12}
		\bigl(\mathcal{N}(\nabla\bu), \vepsi(\bv) \bigr)-\bigl( \xi, \div \bv \bigr)
		&= (\bbf, \bv)+\langle \bbf_1,\bv\rangle
		&&\quad\forall \bv\in \bH^1(\Ome), \\
		\kappa_3 \bigl( \xi, \varphi \bigr) +\bigl(\div\bu, \varphi \bigr)
		&= \kappa_1\bigl(\eta, \varphi \bigr) &&\quad\forall \varphi \in L^2(\Ome), \label{2.13}  \\
		\bigl(\eta_t, \psi \bigr)_{\rm dual}
		+\frac{1}{\mu_f} \bigl(K(\nab (\kappa_1\xi +\kappa_2\eta) &-\rho_f\bg), \nab \psi \bigr) \label{2.14} \\
		&= (\phi, \psi)+\langle \phi_1,\psi\rangle &&\quad\forall \psi \in H^1(\Ome) , \no  \\
		p:=\kappa_1\xi +\kappa_2\eta, \qquad
		&q:=\kappa_1\eta-\kappa_3\xi, && \label{2.15} \\
		%\bu(0) = \bu_0, \qquad p(0) &=p_0, && \label{e2.8} \\
		%q(0)=q_0:=\div \bu_0,\quad \quad
		\eta(0)= \eta_0:&=c_0p_0+\alpha q_0.  && \label{2.16}
	\end{alignat}
\end{definition}

\begin{lemma}\label{lma210903-1}
	There exist positive constants $ C_{1},~C_{2},~C_{3} $ such that
	\begin{align}
		&\left\|\mathcal{N}(\nabla\mathbf{u}) \right\|_{L^{2}(\Omega)}\leq C_{1}\left\|\varepsilon(\mathbf{u}) \right\|_{L^{2}(\Omega)},\label{3.203}\\ 
		&(\mathcal{N}(\nabla(\mathbf{u})),\varepsilon(\mathbf{u}))\geq C_{2}\left\|\varepsilon(\mathbf{u}) \right\|_{L^{2}(\Omega)}^{2},\label{3.204}\\
		&(\mathcal{N}(\nabla(\mathbf{u}))-\mathcal{N}(\nabla(\mathbf{v})),\varepsilon(\mathbf{u})-\varepsilon(\mathbf{v}))\geq C_{4}\left\|\varepsilon(\mathbf{u})-\varepsilon(\mathbf{v}) \right\|_{L^{2}(\Omega)}^{2}.\label{3.207} 
	\end{align}
\end{lemma}
\begin{lemma}\label{lma210903-2}
	There exists positive real number $ C_{3} $ such that the following holds:
	\begin{align}
		&\left\|\mathcal{N}(\nabla\mathbf{u})-\mathcal{N}(\nabla\mathbf{v}) \right\|_{L^{2}(\Omega)}\leq C_{3}\left\|\varepsilon(\mathbf{u})-\varepsilon(\mathbf{v}) \right\|_{L^{2}(\Omega)}.\label{3.206}
	\end{align}
\end{lemma}

As for the detailed proofs of Lemma \ref{lma210903-1} and Lemma \ref{lma210903-2}, one can refer to \cite{201912096}. About the energy estimates of the weak solution and the well-posedness of the weak solution, one can also refer to  \cite{201912096}, here we only list up the main results (see Lemma \ref{estimates}, Lemma \ref{smooth} and Theorem \ref{thm2.5}) as follows:
\begin{lemma}\label{estimates}
	There exists a positive constant
	$ \acute{C}_1=\acute{C}_1\bigl(\|\bu_0\|_{H^1(\Ome)}, \|p_0\|_{L^2(\Ome)},$
	$\|\bbf\|_{L^2(\Ome)},\|\bbf_1\|_{L^2(\p \Ome)},\|\phi\|_{L^2(\Ome)}, \|\phi_1\|_{L^2(\p\Ome)} \bigr)$
	such that
	\begin{align}
		&\sqrt{C_{2}}\|\varepsilon(\bu)\|_{L^\infty(0,T;L^2(\Ome))}
		+\sqrt{\kappa_2} \|\eta\|_{L^\infty(0,T;L^2(\Ome))} \\
		&\qquad
		+\sqrt{\kappa_3} \|\xi\|_{L^\infty(0,T;L^2(\Ome))}
		+\sqrt{\frac{K_1}{\mu_f}} \|\nab p \|_{L^2(0,T;L^2(\Ome))} \leq \acute{C}_1, \no \\
		&\|\bu\|_{L^\infty(0,T;L^2(\Ome))}\leq \acute{C}_1, \quad
		\|p\|_{L^\infty(0,T;L^2(\Ome))} \leq \acute{C}_1 \bigl( \kappa_2^{\frac12} + \kappa_1 \kappa_3^{-\frac12}
		\bigr), \\
		&\|p\|_{L^2(0,T; L^2(\Ome))} \leq \acute{C}_1,~ \quad
		\|\xi\|_{L^2(0,T;L^2(\Ome))} \leq \acute{C}_1\kappa_1^{-1} \bigl(1+ \kappa_2^{\frac12} \bigr).
	\end{align}
\end{lemma}
\begin{lemma}\label{smooth}
	Suppose that $\bu_0$ and $p_0$ are sufficiently smooth, then
	there exist positive constants~ $ \acute{C}_2=\acute{C}_2\bigl(\acute{C}_1,\|\nab p_0\|_{L^2(\Ome)} \bigr)$
	and~ $ \acute{C}_3=\acute{C}_3\bigl(\acute{C}_1,\acute{C}_2, \|\bu_0\|_{H^2(\Ome)},\|p_0\|_{H^2(\Ome)} \bigr)$ such that
	\begin{align}
		&\sqrt{C_{2}}\|\varepsilon(\bu_{t})\|_{L^2(0,T;L^2(\Ome))}
		+\sqrt{\kappa_2} \|\eta_t\|_{L^2(0,T;L^2(\Ome))} \\
		&\qquad
		+\sqrt{\kappa_3} \|\xi_t\|_{L^2(0,T;L^2(\Ome))}
		+\sqrt{\frac{K_1}{\mu_f}} \|\nab p \|_{L^\infty(0,T;L^2(\Ome))} \leq \acute{C}_2, \no \\
		&\sqrt{C_{2}}\|\varepsilon(\bu_{t})\|_{L^\infty(0,T;L^2(\Ome))}
		+\sqrt{\kappa_2} \|\eta_t\|_{L^\infty(0,T;L^2(\Ome))} \\
		&\qquad
		+\sqrt{\kappa_3} \|\xi_t\|_{L^\infty(0,T;L^2(\Ome))}
		+\sqrt{\frac{K_1}{\mu_f}} \|\nab p_t \|_{L^2(0,T;L^2(\Ome))} \leq \acute{C}_3, \no \\
		&\|\eta_{tt}\|_{L^2(H^{1}(\Ome)')} \leq \sqrt{\frac{K_2}{\mu_f}}\acute{C}_3. 
	\end{align}
\end{lemma}
\begin{theorem}\label{thm2.5}
	Let $\bu_0\in\bH^1(\Ome),~ \bbf\in\bL^2(\Omega),~
	\bbf_1\in \bL^2(\p\Ome),~ p_0\in L^2(\Ome),~ \phi\in L^2(\Ome)$ and $\phi_1\in L^2(\p\Ome)$. Suppose $c_0>0$ and $(\bbf,\bv)+\langle \bbf_1, \bv \rangle =0$ for any $\bv\in \mathbf{RM}$. Then there exists a unique solution to the problem \reff{8.71}-\reff{8.77} in the sense of Definition \ref{weak1}. Likewise, there exists a unique solution to the problem \reff{2.3}-\reff{2.5} with \reff{2.9}-\reff{2.11} in the sense of Definition \ref{weak2}.
\end{theorem}

\section{Fully discrete multiphysics finite element method}\label{sec-3}
\subsection{Formulation of fully discrete finite element method}\label{sec-3.1}
 Let $\mathcal{T}_h$ be a 
quasi-uniform triangulation or rectangular partition of $\Omega$ with maximum mesh size $h$, and $\bar{\Omega}=\bigcup_{\mathcal{K}\in\mathcal{T}_h}\bar{\mathcal{K}}$. The time interval $[0, T]$ is divided into $N$ equal intervals, denoted by $[t_{n-1}, t_{n}], n=1,2,...N$,  and $\Delta t=\frac{T}{N}$, then $t_n=n\Delta t$. In this work, we use backward Euler method and denote $ d_{t} v^{n}:=\frac{v^{n}-v^{n-1}}{\Delta t}$. 

Also, let $(\bX_h, M_h)$ be a stable mixed finite element pair, that is, $\bX_h\subset \bH^1(\Omega)$ and $M_h\subset L^2(\Omega)$ 
satisfy the inf-sup condition
\begin{alignat}{2}\label{3.1}
	\sup_{\bv_h\in \bX_h}\frac{({\rm div} \bv_h, \varphi_h)}{\|\bv_h\|_{H^1(\Ome)}}
	\geq \beta_0\|\varphi_h\|_{L^2(\Ome)} &&\quad \forall\varphi_h\in M_{0h}:=M_h\cap L_0^2(\Omega),\ \beta_0>0.
\end{alignat}

A number of stable mixed finite element spaces $(\bX_h, M_h)$ have been known in the literature
\cite{brezzi}. A well-known example is the following
so-called Taylor-Hood element (cf. \cite{ber,brezzi}):
\begin{align*}
	\bX_h &=\bigl\{\bv_h\in \bC^0(\overline{\Ome});\,
	\bv_h|_\mathcal{K}\in \bP_2(\mathcal{K})~~\forall \mathcal{K}\in \cT_h \bigr\}, \\
	M_h &=\bigl\{\varphi_h\in C^0(\overline{\Ome});\, \varphi_h|_\mathcal{K}\in P_1(\mathcal{K})
	~~\forall \mathcal{K}\in \cT_h \bigr\}.
\end{align*}

Finite element approximation space $W_h$ for $\eta$ variable can be chosen independently, any piecewise polynomial space is acceptable provided that
$W_h \supset M_h$, the most convenient choice is $W_h =M_h$.

Define
\begin{equation}\label{3.2}
	\bV_h:=\bigl\{\bv_h\in \bX_h;\,  (\bv_h,\br)=0\,\,
	\forall \br\in \bRM \bigr\},
\end{equation}
it is easy to check that $\bX_h=\bV_h\bigoplus \bRM$. It was proved in \cite{fh10}
that there holds the following inf-sup condition:
\begin{align}\label{3.3}
	\sup_{\bv_h\in \bV_h}\frac{(\div\bv_h,\varphi_h)}{\norm{\bv_h}{H^1(\Ome)}} 
	\geq \beta_1 \norm{\varphi_h}{L^2(\Ome)} \quad \forall \varphi_h\in M_{0h}, \quad \beta_1>0.
\end{align}

Also, we recall the following inverse inequality for polynomial functions \cite{bs08,brezzi,cia} :
\begin{align}
	\left\|\nabla\varphi_{h} \right\|_{L^{2}(\mathcal{K})}\leq c_{1}h^{-1}\left\|\varphi_{h} \right\|_{L^{2}(\mathcal{K})}~~~~~\forall\varphi_{h}\in P_{r}(\mathcal{K}), \mathcal{K}\in T_{h}.\label{20210915} 
\end{align}

Now, we give the fully discrete multiphysics finite element algorithm for the problem \reff{2.3}-\reff{2.5}.

{\bf Multiphysics Finite Element Algorithm (MFEA)} 

\begin{itemize}
	\item[(i)]
	Compute $\bu^0_h\in \bV_h$ and $q^0_h\in W_h$ by $\bu^0_h =\bu_0, p^0_h =p_0$.
%	\begin{eqnarray*}
%		&&\bu^0_h =\bu_0, \quad p^0_h =p_0, \\
%		&&\eta^0_h =c_0p^0_h+\alpha q^0_h, \quad
%		\xi_h^0 =\alpha p_h^0 -\lambda q_h^0.
%	\end{eqnarray*}
%	where $\mathcal{Q}_h$ denotes the $L^2$-projection operator (see \eqref{eq210607-1}).
	
	\item[(ii)] For $n=0,1,2, \cdots$,  do the following two steps.
	
	{\em Step 1:} Solve for $(\bu^{n+1}_h,\xi^{n+1}_h, \eta^{n+1}_h)\in \bV_h\times M_h \times  W_h$ such that
	\begin{eqnarray}
		&\quad\bigl(\mathcal{N}(\nabla\bu^{n+1}_h), \vepsi(\bv_h) \bigr)-\bigl( \xi^{n+1}_h, \div \bv_h \bigr)
		= (\bbf, \bv_h)+\langle \bbf_1,\bv_h\rangle & \forall \bv_h\in \bV_h,\label{3.4} \\
		&\kappa_3\bigl(\xi^{n+1}_h, \varphi_h \bigr) +\bigl(\div\bu^{n+1}_h, \varphi_h \bigr)
		=\kappa_1\bigl( \eta^{n+\theta}_h, \varphi_h \bigr)
		& \forall \varphi_h \in M_h, \label{3.5} \\
		&\bigl(d_t\eta^{n+1}_h, \psi_h \bigr)
		+\frac{1}{\mu_f} \bigl(K(\nab (\kappa_1\xi^{n+1}_h +\kappa_2\eta^{n+1}_h) &\label{3.6}\\
		&\hskip 1in
		-\rho_f\bg,\nab\psi_h \bigr)=(\phi, \psi_h)+\langle \phi_1,\psi_h\rangle &  \forall \psi_h\in W_h,  \no
	\end{eqnarray}
	where $\theta=0$ or $1$.
	
	{\em Step 2:} Update $p^{n+1}_h$ and $q^{n+1}_h$ by
	\begin{alignat}{2}
		p^{n+1}_h=\kappa_1\xi^{n+1}_h +\kappa_2\eta^{n+\theta}_h, \quad
		q^{n+1}_h=\kappa_1\eta^{n+1}_h-\kappa_3\xi^{n+1}_h.\label{3.7}
	\end{alignat}
\end{itemize}

\begin{remark}\label{rem3.1}
The Newton's iterative method to solve the nonlinear Stokes problem \reff{3.4}-\reff{3.5} when $\theta=0$ is
\begin{align}\label{3.8}
		& - (\lambda\nabla\bu_{h}^{n}:\nabla\bu_{h}^{n}\mathbf{I},\varepsilon(\bv_{h}))+ (\bbf, \bv_h)+\langle \bbf_1,\bv_h\rangle+\kappa_1\bigl( \eta^{n+\theta}_h, \varphi_h \bigr)\\
		&=\mu(\varepsilon(\bu_{h}^{n+1}),\varepsilon(\bv_{h}))
		+\mu(\nabla^{T}\bu_{h}^{n+1}\nabla\bu_{h}^{n},\varepsilon(\bv_{h}))
		+\mu(\nabla^{T}\bu_{h}^{n}\nabla\bu_{h}^{n+1},\varepsilon(\bv_{h}))\no\\
		&~  -\mu(\nabla^{T}\bu_{h}^{n}\nabla\bu_{h}^{n},\varepsilon(\bv_{h}))
		+(2\lambda\nabla\bu_{h}^{n}:(\nabla\bu_{h}^{n+1}-\nabla\bu_{h}^{n})\mathbf{I},\varepsilon(\bv_{h}))\no\\
		&~  -(\xi_{h}^{n+1}, \div \bv_{h})+\kappa_{3}(\xi_{h}^{n+1},\varphi_{h})
		+(\div\bu_{h}^{n+1},\varphi_{h}).\no
	\end{align}
\end{remark}
\begin{remark}\label{rem2110-1}
	As for the multiphysics finite element algorithm, the original pressure is eliminated in the reformulation, which will be helpful to overcome the ``locking phenomenon", the later numerical tests show that our proposed method has no ``locking phenomenon", one can see Section \ref{sec-4}.
\end{remark}

%%%%%%%
\subsection{Stability analysis}
The primary goal of this subsection is to derive a discrete energy law which mimics 
the PDE energy law \cite{201912096}. Before discussing the stability of (MFEA), we first show that the numerical solution satisfies the following constraints which are fulfilled by the PDE solution. %{\color{red}{
\begin{lemma}\label{lma3.1}
	Let $\{(\bu_h^{n}, \xi_h^{n}, \eta_h^{n})\}_{n\geq 0}$ be defined by the (MFEA), then there hold
	\begin{alignat}{2}
		(\eta^{n}_h, 1)&=C_\eta(t_n) &&\qquad\mbox{for } n=0, 1, 2, \cdots,\label{e3.11}\\
		(\xi^{n}_h, 1)&=C_\xi(t_{n-1+\theta}) &&\qquad\mbox{for }  n=1-\theta, 1, 2, \cdots, \label{e3.12} \\
		\langle\bu^{n}_h\cdot\bn, 1\rangle &=C_{\bu}(t_{n-1+\theta})
		&&\qquad\mbox{for } n=1-\theta, 1, 2, \cdots.\label{e3.13}
	\end{alignat}
\end{lemma}
\begin{proof}
Taking $\psi_h=1$ in \reff{3.6}, we have
\begin{eqnarray}
\bigl(d_t\eta_h^{n+1}, 1\bigr) =(\phi,1) + \langle \phi_1, 1\rangle.\label{eq210906-1}
\end{eqnarray}
Summing \reff{eq210906-1} over $n$ from $0$ to $\ell \,(\geq 0)$, we get
\begin{eqnarray}
(\eta_h^{\ell+1},1)=(\eta_h^{0},1) + \bigl[(\phi,1) + \langle \phi_1, 1\rangle\bigr] t_{\ell+1}\\
=(\eta_0,1) + \bigl[(\phi,1) + \langle \phi_1, 1\rangle\bigr] t_{\ell+1}
=C_{\eta}(t_{\ell+1})\no 
\end{eqnarray}
for $\ell=0,1,2,\cdots$, which implies that \reff{e3.11} holds.

Taking $\bv_h=\bx$ in \reff{3.4} and $\varphi_h=1$ in \reff{3.5}, we get
\begin{align}
\bigl(\mathcal{N}(\nabla\bu_{h}^{n+1}), \mathbf{I}\bigr) -d\bigl( \xi_h^{n+1}, 1\bigr) 
&=\bigl( \bbf, \bx\bigr)+ \langle \bbf_1,\bx\rangle,\label{add_1}\\
\kappa_3\bigl(\xi_h^{n+1}, 1\bigr) +\bigl( \div \bu_h^{n+1}, 1\bigr) &=\kappa_1 C_\eta(t_{n+\theta}). \label{add_2}
\end{align}
Substituting \eqref{add_2} into \eqref{add_1}£¬ we have
\[
 \bigl(\xi_h^{n+1}, 1\bigr) = \dfrac{1}{d-\kappa_{3}}\left[\bigl(\mathcal{N}(\nabla\bu_{h}^{n+1}), \mathbf{I}\bigr)+\bigl( \div \bu_h^{n+1}, 1\bigr)-\kappa_{1}C_{\eta}(t_{n+\theta})-(\bbf,\bx) -\left\langle\bbf_{1},\bx\right\rangle\right].  
\]
Hence, by the definition of $C_\xi(t)$ in \cite{201912096}, we conclude that \reff{e3.12} holds for all $n\geq 1-\theta$. 

Using \reff{e3.11}, \reff{e3.12}, \eqref{add_2} and Gauss divergence theorem, we deduce that \reff{e3.13} holds.  The proof is complete.
\end{proof}

\begin{lemma}\label{lma3.3}
	Let $ \left\lbrace (\bu_{h}^{n},\xi_{h}^{n},\eta_{h}^{n})\right\rbrace_{n\geq0}  $ be defined by the (MFEA), then there holds the following inequality:
	\begin{align}
		J_{h,\theta}^{l}+S_{h,\theta}^{l}\leq J_{h,\theta}^{0}~~~~~~~~~~for~l\geq1,~\theta=0,1, \label{3.23}
	\end{align}
	where
	\begin{align*}
		&J_{h,\theta}^{l}:=\dfrac{1}{2}\left[C_{2}\left\|\varepsilon(\bu_{h}^{l+1}) \right\|_{L^{2}(\varOmega)}^{2}+\kappa_{2}\left\|\eta_{h}^{l+\theta}\right\|_{L^{2}(\varOmega)}^{2}+\kappa_{3}\left\|\xi_{h}^{l+1}\right\|_{L^{2}(\varOmega)}^{2}\right.\\
		&\left.-2(\bbf,\bu_{h}^{l+1})-2\left\langle\bbf_{1},\bu_{h}^{l+1} \right\rangle \right],\\
		&S_{h,\theta}^{l}:=\varDelta t\sum_{n=1}^{l}\left[\dfrac{\varDelta t}{2}C_{4}\left\|d_{t}\varepsilon(\bu_{h}^{n+1})\right\|_{L^{2}(\varOmega)}^{2}+\dfrac{1}{\mu_{f}}(K\nabla p_{h}^{n+1}-K\rho_{f}g,\nabla p_{h}^{n+1})\right.\\
		&+\dfrac{\kappa_{2}\varDelta t}{2}\left\|d_{t}\eta_{h}^{n+\theta}\right\|_{L^{2}(\varOmega)}^{2}+\dfrac{\kappa_{3}\varDelta t}{2}\left\|d_{t}\xi_{h}^{n+1}\right\|_{L^{2}(\varOmega)}^{2} -(\phi,p_{h}^{n+1})-\left\langle\phi_{1},p_{h}^{n+1}\right\rangle \\
		&\left.-(1-\theta)\dfrac{\kappa_{1}\varDelta t}{\mu_{f}}(Kd_{t}\nabla\xi_{h}^{n+1},\nabla p_{h}^{n+1}) -\dfrac{C_{1}}{\varDelta t}\left\|\varepsilon(\bu_{h}^{n}) \right\|_{L^{2}(\varOmega)}\left\|\varepsilon(\bu_{h}^{n+1}) \right\|_{L^{2}(\varOmega)}\right]. 
	\end{align*}
\end{lemma}
\begin{proof}
	(i) when $ \theta=0 $, based on \reff{3.5}, we can  define $ \eta_{h}^{-1} $ by
	\begin{align}
		\kappa_{1}(\eta_{h}^{-1},\varphi_{h})=\kappa_{3}(\xi_{h}^{0},\varphi_{h})+(\nabla\cdot\bu_{h}^{0},\varphi_{h}).\label{3.24}
	\end{align}
	Setting $ \bv_{h}=d_{t}\bu_{h}^{n+1} $ in \reff{3.4}, we have
	\begin{eqnarray}
		&&\qquad(\mathcal{N}(\nabla\mathbf{u}_{h}^{n+1}),\varepsilon(d_{t}\bu_{h}^{n+1}))-(\xi_{h}^{n+1},\nabla\cdot d_{t}\bu_{h}^{n+1})=\left(\bbf,d_{t}\bu_{h}^{n+1}\right)+\left\langle \bbf_{1},d_{t}\bu_{h}^{n+1}\right\rangle.\label{3.25}
	\end{eqnarray}
Using \reff{3.25}, \reff{3.207}, the  Cauchy-Schwarz inequality, \reff{3.203} and \reff{3.204}, we have
	\begin{align}
		&\dfrac{1}{2\varDelta t}\left[C_{4}\varDelta t^{2}\left\|d_{t}\varepsilon(\bu_{h}^{n+1})\right\|_{L^{2}(\varOmega)}^{2}-2C_{1}\left\|\varepsilon(\bu_{h}^{n}) \right\|_{L^{2}(\varOmega)}\left\|\varepsilon(\bu_{h}^{n+1}) \right\|_{L^{2}(\varOmega)}\right.\label{3.26}\\
		&\left.+\varDelta tC_{2}d_{t}\left\|\varepsilon(\bu_{h}^{n+1}) \right\|_{L^{2}(\varOmega)}^{2}\right]-(\xi_{h}^{n+1},\nabla\cdot d_{t}\bu_{h}^{n+1})\nonumber\\
		&\leq\left(\bbf,d_{t}\bu_{h}^{n+1}\right) +\left\langle \bbf_{1},d_{t}\bu_{h}^{n+1}\right\rangle.\nonumber
	\end{align}
	Setting $ \varphi_{h}=\xi_{h}^{n+1} $ in\reff{3.5}, we get
	\begin{align}
		\kappa_{3}(d_{t}\xi_{h}^{n+1},\xi_{h}^{n+1})+(\nabla\cdot d_{t}\bu_{h}^{n+1},\xi_{h}^{n+1})=\kappa_{1}(d_{t}\eta_{h}^{n},\xi_{h}^{n+1}).\label{3.27}
	\end{align}
	Setting $ \psi_{h}=p_{h}^{n+1} $ in \reff{3.6}  after lowing the super-index from $ n+1 $ to $ n $ on both sides of \reff{3.7}, we get
	\begin{align}
		&(d_{t}\eta_{h}^{n},p_{h}^{n+1})+\dfrac{1}{\mu_{f}}(K(\nabla(\kappa_{1}\xi_{h}^{n}+\kappa_{2}\eta_{h}^{n})-\rho_{f}\bg),\nabla p_{h}^{n+1})\label{3.28}\\
		&=(\phi,p_{h}^{n+1})+\left\langle\phi_{1},p_{h}^{n+1} \right\rangle.\nonumber
	\end{align}
	
	Using the fact of 
	$(d_{t}\xi_{h}^{n+1},\xi_{h}^{n+1})%\dfrac{1}{\varDelta t}((\xi_{h}^{n+1}-\xi_{h}^{n}),\xi_{h}^{n+1})\label{3.29}\\
		%&=\dfrac{1}{2\varDelta t}\left[2(\xi_{h}^{n+1},\xi_{h}^{n+1})-2(\xi_{h}^{n},\xi_{h}^{n+1})+(\xi_{h}^{n},\xi_{h}^{n})-(\xi_{h}^{n},\xi_{h}^{n})) \right]\nonumber\\
	%	&=\dfrac{1}{2\varDelta t}\left[(\xi_{h}^{n+1}-\xi_{h}^{n},\xi_{h}^{n+1}-\xi_{h}^{n})+(\xi_{h}^{n+1},\xi_{h}^{n+1})-(\xi_{h}^{n},\xi_{h}^{n})\right]\nonumber\\
	%	&= \dfrac{\varDelta t}{2}(\dfrac{\xi_{h}^{n+1}-\xi_{h}^{n}}{\varDelta t},\dfrac{\xi_{h}^{n+1}-\xi_{h}^{n}}{\varDelta t})+\dfrac{1}{2}\dfrac{(\xi_{h}^{n+1},\xi_{h}^{n+1})-(\xi_{h}^{n},\xi_{h}^{n})}{\varDelta t}\nonumber\\
		=\dfrac{\varDelta t}{2}\left\|d_{t}\xi_{h}^{n+1} \right\|_{L^{2}(\varOmega)}^{2}+\dfrac{1}{2}d_{t}\left\|\xi_{h}^{n+1} \right\|_{L^{2}(\varOmega)}^{2}$,
	we can rewrite \reff{3.27}  as
	\begin{align}
		&\dfrac{\kappa_{3}\varDelta t}{2}\left\|d_{t}\xi_{h}^{n+1} \right\|_{L^{2}(\varOmega)}^{2}+\dfrac{\kappa_{3}}{2}d_{t}\left\|\xi_{h}^{n+1} \right\|_{L^{2}(\varOmega)}^{2}  +(\nabla\cdot d_{t}\bu_{h}^{n+1},\xi_{h}^{n+1})\label{3.30}\\
		&=\kappa_{1}(d_{t}\eta_{h}^{n},\xi_{h}^{n+1}).\nonumber
	\end{align}
	Similarly,  we get
	\begin{align}
		&\dfrac{\kappa_{2}\varDelta t}{2}\left\|d_{t}\eta_{h}^{n} \right\|_{L^{2}(\varOmega)}^{2}+\dfrac{\kappa_{2}}{2}d_{t}\left\|\eta_{h}^{n} \right\|_{L^{2}(\varOmega)}^{2}+\kappa_{1}(d_{t}\eta_{h}^{n},\xi_{h}^{n+1})\label{3.31}\\
		&+\dfrac{1}{\mu_{f}}(K(\nabla p_{h}^{n+1}-\rho_{f}\bg),\nabla p_{h}^{n+1})-\dfrac{\kappa_{1}\varDelta t}{\mu_{f}}(Kd_{t}\nabla\xi_{h}^{n+1},\nabla p_{h}^{n+1})\nonumber\\
		&=(\phi,p_{h}^{n+1})+\left\langle\phi_{1},p_{h}^{n+1} \right\rangle.\nonumber
	\end{align} 
	Combining \reff{3.26}, \reff{3.30} and \reff{3.31}, we see that \reff{3.23}  holds for $ \theta=0$ if $ \Delta t=O(h^{2})$.
	
	(ii) When $\theta=1$, setting $ \varphi_{h}=\xi_{h}^{n+1} $ in \reff{3.5}  and $ \psi_{h}=p_{h}^{n+1} $ in \reff{3.6}, we get
%	\begin{align}
%		&\kappa_{3}(d_{t}\xi_{h}^{n+1},\xi_{h}^{n+1})+(\nabla\cdot d_{t}\bu_{h}^{n+1},\xi_{h}^{n+1})=\kappa_{1}(d_{t}\eta_{h}^{n+1},\xi_{h}^{n+1}),\label{3.32}\\
%		&(d_{t}\eta_{h}^{n+1},p_{h}^{n+1})+\dfrac{1}{\mu_{f}}(K(\nabla(\kappa_{1}\xi_{h}^{n+1}+\kappa_{2}\eta_{h}^{n+1})-\rho_{f}\bg),\nabla p_{h}^{n+1})\label{3.33}\\
%		&=(\phi,p_{h}^{n+1})+\left\langle\phi_{1},p_{h}^{n+1} \right\rangle.\nonumber
%	\end{align}
%	Using \reff{3.29}, we have
	\begin{align}
		&\dfrac{\kappa_{3}\varDelta t}{2}\left\|d_{t}\xi_{h}^{n+1} \right\|_{L^{2}(\varOmega)}^{2}+\dfrac{\kappa_{3}}{2}d_{t}\left\|\xi_{h}^{n+1} \right\|_{L^{2}(\varOmega)}^{2} +(\nabla\cdot d_{t}\bu_{h}^{n+1},\xi_{h}^{n+1})\label{3.34}\\
		&=\kappa_{1}(d_{t}\eta_{h}^{n+1},\xi_{h}^{n+1}),\no\\
		&\dfrac{\kappa_{2}\varDelta t}{2}\left\|d_{t}\eta_{h}^{n+1} \right\|_{L^{2}(\varOmega)}^{2}+\dfrac{\kappa_{2}}{2}d_{t}\left\|\eta_{h}^{n+1} \right\|_{L^{2}(\varOmega)}^{2}+\kappa_{1}(d_{t}\eta_{h}^{n+1},\xi_{h}^{n+1})\label{3.35}\\
		&+\dfrac{1}{\mu_{f}}(K(\nabla p_{h}^{n+1}-\rho_{f}\bg),\nabla p_{h}^{n+1})=(\phi,p_{h}^{n+1})+\left\langle\phi_{1}, p_{h}^{n+1} \right\rangle.\nonumber
	\end{align}
	Combining \reff{3.26}, \reff{3.34} and \reff{3.35}, we imply that \reff{3.23} holds. The proof is complete.
\end{proof}
\begin{lemma}
	Let $ \left\lbrace (\bu_{h}^{n},\xi_{h}^{n},\eta_{h}^{n})\right\rbrace_{n\geq0}  $ be defined by the (MFEA) with $ \theta=0 $, then there holds the following inequality:
	\begin{align}
		J_{h,\theta}^{l}+\hat{S}_{h,0}^{l}\leq J_{h,\theta}^{0}~~~~~~~~~~for~l\geq1 \label{3.36}
	\end{align}
	provided that $ \varDelta t=O(h^{2}) $. Here
	\begin{align*}
		&\hat{S}_{h,0}^{l}:=\varDelta t\sum_{n=1}^{l}\left[\dfrac{\varDelta t}{4}C_{4}\left\|d_{t}\varepsilon(\bu_{h}^{n+1})\right\|_{L^{2}(\varOmega)}^{2}+\dfrac{K}{2\mu_{f}}\left\|\nabla p_{h}^{n+1} \right\|_{L^{2}(\varOmega)}^{2}\right.\\
		&-\dfrac{K}{\mu_{f}}(\rho_{f}g,\nabla p_{h}^{n+1})+\dfrac{\kappa_{2}\varDelta t}{2}\left\|d_{t}\eta_{h}^{n+\theta}\right\|_{L^{2}(\varOmega)}^{2}+\dfrac{\kappa_{3}\varDelta t}{2}\left\|d_{t}\xi_{h}^{n+1}\right\|_{L^{2}(\varOmega)}^{2}\\
		&\left.-(\phi,p_{h}^{n+1})-\left\langle\phi_{1},p_{h}^{n+1}\right\rangle-\dfrac{C_{1}}{\varDelta t}\left\|\varepsilon(\bu_{h}^{n+1}) \right\|_{L^{2}(\varOmega)}\left\|\varepsilon(\bu_{h}^{n+1}) \right\|_{L^{2}(\varOmega)}\right].
	\end{align*}
\end{lemma}
\begin{proof}
Using the Cauchy-Schwarz inequality and inverse inequality \reff{20210915}, we get 
	\begin{align}
		&\dfrac{K\kappa_{1}\varDelta t}{\mu_{f}}(d_{t}\nabla\xi_{h}^{n+1},\nabla p_{h}^{n+1})\label{3.37}\\
		&\leq \dfrac{K\kappa_{1}^{2}}{2\mu_{f}}\left\|\nabla\xi_{h}^{n+1}-\nabla\xi_{h}^{n} \right\|_{L^{2}(\varOmega)}^{2}+\dfrac{K}{2\mu_{f}}\left\|\nabla p_{h}^{n+1} \right\|_{L^{2}(\varOmega)}^{2}\no\\
		&\leq\dfrac{K\kappa_{1}^{2}c_{1}^{2}}{2\mu_{f}h^{2}}\left\|\xi_{h}^{n+1}-\xi_{h}^{n} \right\|_{L^{2}(\varOmega)}^{2}+\dfrac{K}{2\mu_{f}}\left\|\nabla p_{h}^{n+1} \right\|_{L^{2}(\varOmega)}^{2}.\nonumber 
	\end{align}
	To bound the first term on the right-hand side of \reff{3.37}, we appeal to the inf-sup condition \reff{3.3} and get
	\begin{align}
		\left\|\xi_{h}^{n+1}-\xi_{h}^{n}\right\|_{L^{2}(\varOmega)}&\leq\dfrac{1}{\beta_{1}}\sup_{\bv_{h}\in\bV_{h}}\dfrac{(\nabla\cdot\bv_{h},\xi_{h}^{n+1}-\xi_{h}^{n})}{\left\|\nabla\bv_{h} \right\|_{L^{2}(\varOmega)} }\label{3.38}\\
		&\leq\dfrac{1}{\beta_{1}}\sup_{\bv_{h}\in\bV_{h}}\dfrac{(\mathcal{N}(\nabla\mathbf{u}_{h}^{n+1})-\mathcal{N}(\nabla\mathbf{u}_{h}^{n}),\varepsilon(\bv_{h}))}{\left\|\nabla\bv_{h} \right\|_{L^{2}(\varOmega)} }\nonumber\\
		&\leq\frac{C_{3}\varDelta t}{\beta_{1}}\left\|d_{t}\varepsilon(\bu^{n+1}_{h}) \right\|_{L^{2}(\varOmega)}.\nonumber 
	\end{align}
	Substituting \reff{3.38} into \reff{3.37} and combining it with \reff{3.23} imply that \reff{3.36} holds if $\varDelta t\leq\dfrac{C_{4}\mu_{f}\beta_{1}^{2}h^{2}}{2C_{3}^{2}\kappa_{1}^{2}Kc_{1}^{2}}$. The proof is complete.
\end{proof}
\begin{theorem}
	The numerical solution $ \left\lbrace\bu_{h}^{n+1},\xi_{h}^{n+1},\eta_{h}^{n+1} \right\rbrace_{n\geq0}  $ of the problem \reff{3.4}-\reff{3.6} exists uniquely.
\end{theorem}
\begin{proof}
 Given a function $ \bu_{h}^{n}\in \bV_{h} $, supposing that $U\subset\bV_{h}$ is a compact and convex subspace, setting $g(t):=-\mu\nabla^{T}\bu_{h}^{n+1}\nabla\bu_{h}^{n+1}-\lambda\left\|\nabla\bu_{h}^{n+1} \right\|\mathbf{I}(0\leq t\leq T)$, we have
	\begin{eqnarray}\label{6.9}
		&&\qquad\left\|\mu\nabla^{T}\mathbf{u}_{h}^{n+1}\nabla\mathbf{u}_{h}^{n+1}+\lambda\left\|\nabla\mathbf{u}_{h}^{n+1} \right\|_{F}^{2}\mathbf{I}\right\|_{L^{2}(\varOmega)}\leq\mu\left\|\nabla\mathbf{u}_{h}^{n+1}\right\|_{L^{2}(\varOmega)}^{2}+\lambda dN^{'2}\\
		&&\leq(\mu N+\dfrac{\lambda dN^{'2}}{M})\left\|\nabla\mathbf{u}_{h}^{n+1}\right\|_{L^{2}(\varOmega)},\no
	\end{eqnarray}
which implies that $ g\in L^{2}(0,T;L^{2}(\varOmega))$. 

Next, we consider the linear problem: find $ (w_{h}^{n+1},\xi_{h}^{n+1}, \eta_{h}^{n+1})\in\bV_{h}\times M_{h}\times W_{h}$ satisfying 
	\begin{eqnarray}
		&&\mu\bigl(\vepsi(\bw_{h}^{n+1}), \vepsi(\bv_h) \bigr)-\bigl( \xi^{n+1}_h, \div \bv_h \bigr)
		= (g, \bv_h)\label{6.10}\\
		&&\quad+(f,\bv_{h})+\langle \bbf_1,\bv_h\rangle ,\forall \bv_h\in \bV_h, \no\\
		&&\kappa_3\bigl(\xi^{n+1}_h, \varphi_h \bigr) +\bigl(\div\bu^{n+1}_h, \varphi_h \bigr)
		=\kappa_1\bigl( \eta^{n+\theta}_h, \varphi_h \bigr), \forall \varphi_h \in M_h, \label{6.11}\\
		&&\bigl(d_t\eta^{n+1}_h, \psi_h \bigr)
		+\frac{1}{\mu_f} \bigl(K(\nab (\kappa_1\xi^{n+1}_h +\kappa_2\eta^{n+1}_h) \label{6.12}\\
		&&\quad-\rho_f\bg,\nab\psi_h \bigr)=(\phi, \psi_h)+\langle \phi_1,\psi_h\rangle,~~\forall \psi_h\in W_h.\no  
	\end{eqnarray}

As for \reff{2.4}-\reff{2.5}, according to the theory of linear parabolic equations (cf. \cite{20210820}), we know that $ \xi $ and $ \eta $ can be uniquely determined by $\bw$, that is, $ \exists \varPhi, \varPsi$ such that  $\xi_{h}^{n+1}=\varPhi(\bw_{h}^{n+1})$ and $\eta_{h}^{n+1}=\varPsi(\bw_{h}^{n+1})$. Thus, the problem \reff{6.10}-\reff{6.12} is equivalent to the following problem
\begin{align*}
	\left\lbrace \begin{aligned}	
       &Solve~ for~ \bw_{h}^{n+1}\in\bV_{h}~~ such~ that\\
		&\mu\bigl(\vepsi(\bw_{h}^{n+1}), \vepsi(\bv_h) \bigr)+\bigl( \varPhi(\bw_{h}^{n+1}), \div \bv_h \bigr)
		= (g, \bv_h)+(f,\bv_{h})+\langle \bbf_1,\bv_h\rangle, \forall \bv_h\in \bV_h.
	\end{aligned}\right.
	\end{align*} 
Following the method of \cite{fgl14} or \cite{fglarxiv}, one can prove that the problem \reff{6.10} has a unique solution.
	
Define $ A:\bV_{h}rightarrow \bV_{h} $ by $ A[\bu_{h}^{n+1}]=\bw_{h}^{n+1} $. Similarly, it's easy to know the following problem is equivalent to the problem \reff{3.4}-\reff{3.6}
	\begin{align*}
		\left\lbrace \begin{aligned}
   	    &Solve~for~  \bu_{h}^{n+1}\in\bV_{h}  ~such~ that\\
		&\bigl(\mathcal{N}(\nabla\bu^{n+1}_h), \vepsi(\bv_h) \bigr)+\bigl(\varPhi(\bu_{h}^{n+1}), \div \bv_h \bigr)
		= (\bbf, \bv_h)+\langle \bbf_1,\bv_h\rangle \quad \forall \bv_h\in \bV_h. 
	\end{aligned}\right.
	\end{align*}

Next, we prove that $A$ is continuous. To do that, choose $\bu_{h}^{n+1},\tilde{\bu}_{h}^{n+1} $ and define $ \bw_{h}^{n+1}=A[\bu_{h}^{n+1}],~ \tilde{\bw}_{h}^{n+1}=A[\tilde{\bu}_{h}^{n+1}] $ as above. Consequently $ \bw_{h}^{n+1} $ verifies \reff{6.10}-\reff{6.12} and $ \tilde{\bw}_{h}^{n+1} $ satisfies a similar identity for $ \tilde{g}= -\mu\nabla^{T}\tilde{\bu}_{h}^{n+1}\nabla\tilde{\bu}_{h}^{n+1}-\lambda\left\|\nabla\tilde{\bu}_{h}^{n+1} \right\|\mathbf{I}$. Using \reff{6.10}, Korn's inequality, Poincar$\acute{e}$ inequality and Young inequality, we get
	\begin{eqnarray}
		&\mu\left\| \tilde{\bw}_{h}^{n+1}-\bw_{h}^{n+1}\right\|_{L^{2}(\varOmega)}+c_{1}^{2}(\varPhi(\tilde{\bw}_{h}^{n+1})-\varPhi(\bw_{h}^{n+1}),\tilde{\bw}_{h}^{n+1}-\bw_{h}^{n+1})\label{eq210929-2}\\
		&\leq c_{1}^{2}\mu\left\| \vepsi(\tilde{\bw}_{h}^{n+1})-\vepsi(\bw_{h}^{n+1})\right\|_{L^{2}(\varOmega)}+c_{1}^{2}(\varPhi(\tilde{\bw}_{h}^{n+1})-\varPhi(\bw_{h}^{n+1}),\tilde{\bw}_{h}^{n+1}-\bw_{h}^{n+1})\no\\
		&=c_{1}^{2}(\tilde{g}-g,\tilde{\bw}_{h}^{n+1}-\bw_{h}^{n+1})\no\\
		&\leq c_{1}^{2}\left[ \epsilon\left\| \tilde{\bw}_{h}^{n+1}-\bw_{h}^{n+1}\right\|_{L^{2}(\varOmega)}^{2}+\dfrac{1}{\epsilon}\left\| \tilde{g}-g\right\|_{L^{2}(\varOmega)}\right], \no
	\end{eqnarray}
	where $c_{1}$ is a real positive parameter. Choosing $ \epsilon>0 $ sufficiently small in \reff{eq210929-2}, we have
	\begin{align*}
		c_{1}(\varPhi(\tilde{\bw}_{h}^{n+1})-\varPhi(\bw_{h}^{n+1}),\tilde{\bw}_{h}^{n+1}-\bw_{h}^{n+1})\leq C_{f}\left\| \tilde{g}-g\right\|_{L^{2}(\varOmega)}\leq C_{f}\left\| \tilde{\bu}_{h}^{n+1}-\bu_{h}^{n+1}\right\|_{L^{2}(\varOmega)},
	\end{align*}
	where $C_{f}$ is a real positive number. 
	
It is easy to check that
	\begin{eqnarray}
		\left\| A[\tilde{\bu}_{h}^{n+1}]-A[\bu_{h}^{n+1}]\right\|_{L^{2}(\varOmega)}^{2}\leq \tilde{C}_{f}\left\| \tilde{\bu}_{h}^{n+1}-\bu_{h}^{n+1}\right\|_{L^{2}(\varOmega)}^{2}.\label{eq210929-1}
	\end{eqnarray}
	Using \reff{eq210929-1}, we get
	\begin{align*}
		\left\| A[\tilde{\bu}_{h}^{n+1}]-A[\bu_{h}^{n+1}]\right\|_{L^{2}(\varOmega)}\leq\sqrt{\tilde{C}_{f}}\left\| \tilde{\bu}_{h}^{n+1}-\bu_{h}^{n+1}\right\|_{L^{2}(\varOmega)}.
	\end{align*}
	If $ \tilde{C}_{f} $ is so small, thus $A$ is continuous. Since $U$ is a compact and convex, according to Schauder's fixed point theorem (cf. \cite{20210820}), then $A$ has a fixed point in $U$. 
	
	Next, we prove that the problem \reff{3.4}-\reff{3.6} has the unique solution. Due to $ C_{\eta}(t_{n});=(\eta_{h}^{n},1)=(\eta_{0},1)+\left[(\phi,1)+\left\langle \phi_{1},1\right\rangle  \right] $, using Lemma \ref{lma3.1}, it is easy to check that $\eta$ is unique.
	
Assume that $ (\bu_{h}^{n+1},\xi_{h}^{n+1},\eta_{h}^{n+1}) $ and $ (\tilde{\bu}_{h}^{n+1},\tilde{\xi}_{h}^{n+1},\tilde{\eta}_{h}^{n+1}) $ are the two different solution of the problem \reff{3.4}-\reff{3.6}.

 Using \reff{3.5} and \reff{3.6}, we obtain
	\begin{eqnarray}
		&\qquad(\mathcal{N}(\nabla\bu_{h}^{n+1})-\mathcal{N}(\nabla\tilde{\bu}_{h}^{n+1}),\varepsilon(\bv_{h}))-(\xi_{h}^{n+1}-\tilde{\xi}_{h}^{n+1},\nabla\cdot\bv_{h})=0~~\forall~\bv_{h}\in \bV_{h},\label{2.72}\\
		&\kappa_{3}(\xi_{h}^{n+1}-\tilde{\xi}_{h}^{n+1},\varphi_{h})+(\nabla\cdot\bu_{h}^{n+1}-\nabla\cdot\tilde{\bu}_{h}^{n+1},\varphi_{h})=0~~\forall \varphi_{h}\in W_{h}.\label{2.73}
	\end{eqnarray}
	Adding \reff{2.72} and \reff{2.73}, letting $ \bv_{h}=\bu_{h}^{n+1}-\tilde{\bu}_{h}^{n+1}, \varphi_{h}=\xi_{h}^{n+1}-\tilde{\xi}_{h}^{n+1}$, using \reff{3.207}, we have
	\begin{align}
		0\leq C_{4}\left\|\varepsilon(\bu_{h}^{n+1})-\varepsilon(\tilde{\bu}_{h}^{n+1}) \right\|^{2}_{L^{2}(\varOmega)}+\kappa_{3}\left\|\xi_{h}^{n+1}-\tilde{\xi}_{h}^{n+1} \right\|^{2}_{L^{2}(\varOmega)} =0.\label{eq210907-1}
	\end{align}
	Using \reff{eq210907-1} and the initial value $ \bu_{0}$, we obtain
	\begin{align*}
		\bu_{h}^{n+1}=\tilde{\bu}_{h}^{n+1},~~~\xi_{h}^{n+1}=\xi_{h}^{n+1}.	
	\end{align*}
	Since $ p_{h}^{n+1}=\kappa_{1}\xi_{h}^{n+1}+\kappa_{2}\eta_{h}^{n+1}, q_{h}^{n+1}=\kappa_{1}\eta_{h}^{n+1}-\kappa_{3}\xi_{h}^{n+1}$, so we have
	\begin{align*}
		p_{h}^{n+1}=\tilde{p}_{h}^{n+1},~~~~q_{h}^{n+1}=\tilde{q}_{h}^{n+1}.
	\end{align*}	
	Hence, the assumption is false, so  the problem \reff{3.4}-\reff{3.6} has a  unique weak solution. The proof is complete.
\end{proof}

\section{Error estimates}\label{sec-3.3}
To derive the optimal order error estimates of the fully discrete multiphysics finite element method,
for any $\varphi \in L^{2}(\Omega)$, we firstly define $L^{2}(\Omega)$-projection operators $\mathcal{Q}_{h}: L^{2}(\Omega)\rightarrow X^{k}_{h}$  by
\begin{eqnarray}
	(\mathcal{Q}_{h}\varphi,\psi_{h})=(\varphi,\psi_{h})~~~~~\psi_{h}\in X^{k}_{h},\label{eq210607-1}
\end{eqnarray}
where $X^{k}_{h}:=\{\psi_{h}\in C^{0};~\psi_{h}|_{E}\in P_{k}(E)~\forall E\in\mathcal{T}_{h}\}$.

Next, for any $\varphi\in H^{1}(\Omega)$, we define its elliptic projection $\mathcal{S}_{h}: H^{1}(\Omega)\rightarrow X^{k}_{h}$ by
\begin{align}
	(K\nabla \mathcal{S}_{h}\varphi,&\nabla\varphi_{h})=(K\nabla\varphi,\nabla\varphi_{h})~~~~~\forall \varphi_{h}\in X^{k}_{h},\\
	&(\mathcal{S}_{h}\varphi,1)=(\varphi,1).
\end{align}
Finally, for any $\textbf{v}\in\textbf{H}^{1}(\Omega)$, we define its elliptic projection $\mathcal{R}_{h}: \textbf{H}^{1}(\Omega)\rightarrow\textbf{V}^{k}_{h}$ by
\begin{eqnarray}
	(\varepsilon(\mathcal{R}_{h}\textbf{v}),\varepsilon(\textbf{w}_{h}))=(\varepsilon(\textbf{v}),\varepsilon(\textbf{w}_{h}))~~~~~
	\forall\textbf{w}_{h}\in \textbf{V}^{k}_{h},
\end{eqnarray}
where $\textbf{V}^k_{h}:=\{\textbf{v}_{h} \in \textbf{C}^{0};~\textbf{v}_{h}|_{\mathcal{K}}\in \textbf{P}_{k}(\mathcal{K}), (\textbf{v}_{h}, \textbf{r})=0 ~\forall \textbf{r} \in \textbf{RM}\}$, $k$ is the degree of the piecewise polynomial on $\mathcal{K}$. From \cite{bs08}, we know that  $\mathcal{Q}_{h},\mathcal{S}_{h}$ and $\mathcal{R}_{h}$ satisfy
\begin{align}
	\|\mathcal{Q}_{h}\varphi-\varphi\|_{L^{2}(\Omega)}+&h\|\nabla(\mathcal{Q}_{h}\varphi-\varphi)\|_{L^{2}(\Omega)}\label{3.41}\\
	&\leq Ch^{s+1}\|\varphi\|_{H^{s+1}(\Omega)}~~\forall \varphi\in H^{s+1}(\Omega),~~  0\leq s\leq k,\no\\ \|\mathcal{S}_{h}\varphi-\varphi\|_{L^{2}(\Omega)}+&h\|\nabla(\mathcal{S}_{h}\varphi-\varphi)\|_{L^{2}(\Omega)}\label{3.44}\\ &\leq Ch^{s+1}\|\varphi\|_{H^{s+1}(\Omega)}~~\forall \varphi\in H^{s+1}(\Omega),~~  0\leq s\leq k,\no\\
	\|\mathcal{R}_{h}\textbf{v}-\textbf{v}\|_{L^{2}(\Omega)}+&h\|\nabla(\mathcal{R}_{h}\textbf{v}-\textbf{v})\|_{L^{2}(\Omega)}\label{3.46}\\
	&\leq Ch^{s+1}\|\textbf{v}\|_{H^{s+1}(\Omega)}~~\forall \textbf{v}\in \textbf{H}^{s+1}(\Omega),~~  0\leq s\leq k.\no
\end{align}
To derive the error estimates, we introduce the following notations
\begin{align*}	&E_{\bu}^{n}=\bu(t_{n})-\bu_{h}^{n},~~~E_{\xi}^{n}=\xi(t_{n})-\xi_{h}^{n},~~~~E_{\eta}^{n}=\eta(t_{n})-\eta_{h}^{n},~~~~\\
	&E_{p}^{n}=p(t_{n})-p_{h}^{n},~~~~E_{q}^{n}=q(t_{n})-q_{h}^{n}.	
\end{align*}
It is easy to check out
\begin{align}
	E_{p}^{n}=\kappa_{1}E_{\xi}^{n}+\kappa_{2}E_{\eta}^{n},~~~~E_{q}^{n}=\kappa_{3}E_{\xi}^{n}+\kappa_{1}E_{\eta}^{n}.
\end{align}
Also, we denote
\begin{align*}
	&E_{\bu}^{n}=\bu(t_{n})-\mathcal{R}_{h}(\bu(t_{n}))+\mathcal{R}_{h}(\bu(t_{n}))-\bu_{h}^{n}:=Y_{\bu}^{n}+Z_{\bu}^{n},\\
	&E_{\xi}^{n}=\xi(t_{n})-\mathcal{S}_{h}(\xi(t_{n}))+\mathcal{S}_{h}(\xi(t_{n}))-\xi_{h}^{n}:=Y_{\xi}^{n}+Z_{\xi}^{n},\\
	&E_{\eta}^{n}=\eta(t_{n})-\mathcal{S}_{h}(\eta(t_{n}))+\mathcal{S}_{h}(\eta(t_{n}))-\eta_{h}^{n}:=Y_{\eta}^{n}+Z_{\eta}^{n},\\
	&E_{p}^{n}=p(t_{n})-\mathcal{S}_{h}(p(t_{n}))+\mathcal{S}_{h}(p(t_{n}))-p_{h}^{n}:=Y_{p}^{n}+Z_{p}^{n},\\
	&E_{\xi}^{n}=\xi(t_{n})-\mathcal{Q}_{h}(\xi(t_{n}))+\mathcal{Q}_{h}(\xi(t_{n}))-\xi_{h}^{n}:=F_{\xi}^{n}+G_{\xi}^{n},\\
	&E_{\eta}^{n}=\eta(t_{n})-\mathcal{Q}_{h}(\eta(t_{n}))+\mathcal{Q}_{h}(\eta(t_{n}))-\eta_{h}^{n}:=F_{\eta}^{n}+G_{\eta}^{n},	\\
	&E_{p}^{n}=p(t_{n})-\mathcal{Q}_{h}(p(t_{n}))+\mathcal{S}_{h}(p(t_{n}))-p_{h}^{n}:=F_{p}^{n}+G_{p}^{n}.
\end{align*}
\begin{lemma}
	Let $ \left\lbrace (\bu_{h}^{n}, \xi_{h}^{n}, \eta_{h}^{n})\right\rbrace_{n\geq0}  $ be generated by the (MFEA) and $Y_{\bu}^{n}, Z_{\bu}^{n}, Y_{\xi}^{n}, Z_{\xi}^{n}, Y_{\eta}^{n}$ and $Z_{\eta}^{n}$ be defined as above. Then there holds
	\begin{align}
		&\mathcal{E}_{h}^{l}+\varDelta t\sum_{n=1}^{l}C_{4}\left\|\varepsilon(Z_{\bu}^{n+1}) \right\|_{L^{2}(\varOmega)}^{2}+\varDelta t\sum_{n=1}^{l}\left[\dfrac{K}{\mu_{f}}(\nabla\hat{Z}_{p}^{n+1},\nabla\hat{Z}_{p}^{n+1})\right.\label{3.48}\\
		&\left.+\dfrac{\kappa_{2}\varDelta t}{2}\left\|d_{t}G_{\eta}^{n+\theta}\right\|_{L^{2}(\varOmega)}^{2}+\dfrac{\kappa_{3}\varDelta t}{2}\left\|d_{t}G_{\xi}^{n+1}\right\|_{L^{2}(\varOmega)}^{2}\right] \nonumber\\
		&\leq\mathcal{E}_{h}^{0}+\varDelta t\sum_{n=1}^{l}\left[
		(F_{\xi}^{n+1},\nabla\cdot Z_{\bu}^{n+1})-(\nabla\cdot d_{t}Z_{\bu}^{n+1},G_{\xi}^{n+1})\right]\nonumber\\
		&+\varDelta t\sum_{n=1}^{l}\left[(G_{\xi}^{n+1},\nabla\cdot Z_{\bu}^{n+1})-(\nabla\cdot d_{t}Y_{\bu}^{n+1},G_{\xi}^{n+1}) \right] \nonumber\\
		&+\kappa_{1}(1-\theta)(\varDelta t)^{2}\sum_{n=1}^{l}(d_{t}^{2}\eta(t_{n+1}),G_{\xi}^{n+1})+\varDelta t\sum_{n=1}^{l}(R_{h}^{n+\theta},\hat{Z}_{p}^{n+1})\nonumber\\
		&+\varDelta t\sum_{n=1}^{l}(\mathcal{N}(\nabla\mathcal{R}_{h}\bu(t_{n+1}))-\mathcal{N}(\nabla\bu(t_{n+1})),\varepsilon(Z_{\bu}^{n+1})) \nonumber\\
		&+(1-\theta)(\varDelta t)^{2}\sum_{n=1}^{l}\dfrac{\kappa_{1}}{\mu_{f}}(Kd_{t}\nabla Z_{\xi}^{n+1},\nabla\hat{Z}_{p}^{n+1})\no\\
		&+\varDelta t\sum_{n=1}^{l}(d_{t}G_{\eta}^{n+\theta},Y_{p}^{n+1}-F_{p}^{n+1}),\nonumber
	\end{align}
	where
	\begin{align*}
		&\hat{Z}_{p}^{n+1}:=F_{p}^{n+1}-Y_{p}^{n+1}+\kappa_{1}G_{\xi}^{n+1}+\kappa_{2}G_{\eta}^{n+\theta},\\
		&\mathcal{E}_{h}^{l}:=\dfrac{1}{2}\left[\kappa_{2}\left\|G_{\eta}^{l+\theta} \right\|_{L^{2}(\varOmega)}^{2}+\kappa_{3}\left\|G_{\xi}^{l+1} \right\|_{L^{2}(\varOmega)}^{2}  \right], \\
		&R_{h}^{n+1}:=-\dfrac{1}{\varDelta t}\int_{t_{n}}^{t_{n+1}}(s-t_{n})\eta_{tt}(s)ds.
	\end{align*}
\end{lemma}
\begin{proof}
	Subtracting \reff{3.4} from \reff{2.12}, \reff{3.5} from \reff{2.13}, \reff{3.6}  from \reff{2.14}, respectively, we get
	\begin{eqnarray}
		&&\quad(\mathcal{N}(\nabla\bu(t_{n+1}))-\mathcal{N}(\nabla\mathbf{u}_{h}^{n+1}),\varepsilon(\bv_{h}))-(E_{\xi}^{n+1},\nabla\cdot\bv_{h})=0\quad  \forall~\bv_{h}\in \bV_{h},\\
		&&\quad\kappa_{3}(E_{\xi}^{n+1},\varphi_{h})+(\nabla\cdot E_{\bu}^{n+1},\varphi_{h})=\kappa_{1}(E_{\eta}^{n+\theta},\varphi_{h})\\
		&&\qquad+\kappa_{1}(1-\theta)\varDelta t(d_{t}\eta(t_{n+1}),\varphi_{h}) \quad\forall \varphi_{h}\in M_{h},\no\\
		&&\quad(d_{t}E_{\eta}^{n+\theta},\psi_{h})+\dfrac{1}{\mu_{f}}(K(\nabla E_{p}^{n+1},\nabla\psi_{h}) \\
		&&\quad\quad-(1-\theta)\dfrac{\kappa_{1}\varDelta t}{\mu_{f}}\left(Kd_{t}\nabla E_{\xi}^{n+1},\nabla\psi_{h}  \right)=(R_{h}^{n+\theta},\psi_{h}) \quad\forall\psi_{h}\in W_{h},\no\\
		&&\quad E_{\bu}^{0}=0, E_{\xi}^{0}=0, E_{\eta}^{-1}=0.
	\end{eqnarray}
	Using the definitions of the projection operators $ \mathcal{Q}_{h},\mathcal{S}_{h},\mathcal{R}_{h}$, we have
	\begin{align}
		&(\mathcal{N}(\nabla\bu(t_{n+1}))-\mathcal{N}(\nabla\mathbf{u}_{h}^{n+1}),\varepsilon(\bv_{h}))-(G_{\xi}^{n+1},\nabla\cdot\bv_{h})\label{3.53}\\
		&~~~~~~~~~~~~~~~~~~~~~~~~~~~~~~~~~~~~~=(F_{\xi}^{n+1},\nabla\cdot\bv_{h})~~~~~~~~~~~~\forall~\bv_{h}\in \bV_{h},\no\\
		&\kappa_{3}(G_{\xi}^{n+1},\varphi_{h})+(\nabla\cdot Z_{\bu}^{n+1},\varphi_{h})=\kappa_{1}(G_{\eta}^{n+\theta},\varphi_{h})\label{3.54}\\
		&~~~~~~~-(\nabla\cdot Y_{\bu}^{n+1},\varphi_{h})+\kappa_{1}(1-\theta)\varDelta t(d_{t}\eta(t_{n+1}),\varphi_{h})~~~~\forall \varphi_{h}\in M_{h},\no\\
		&(d_{t}G_{\eta}^{n+\theta},\psi_{h})+\dfrac{1}{\mu_{f}}(K(\nabla\hat{Z}_{p}^{n+1},\nabla\psi_{h})-(1-\theta)\dfrac{\kappa_{1}\varDelta t}{\mu_{f}}\left(Kd_{t}\nabla E_{\xi}^{n+1},\nabla\psi_{h}  \right) \label{3.55}\\
		&~~~~~~~~~~~~~~~~~~~~~~~~~~~~~~~~~~~~=(R_{h}^{n+\theta},\psi_{h})~~~~~~~~~~~~~~~~~~~\forall~\psi_{h}\in W_{h}.\no
	\end{align}
	Setting $ \bv_{h}=Z_{\bu}^{n+1} $ in \reff{3.53}, we have
	\begin{align}
		(\mathcal{N}(\nabla\bu(t_{n+1}))-\mathcal{N}(\nabla\mathbf{u}_{h}^{n+1}),\varepsilon(Z_{\bu}^{n+1}))&-(G_{\xi}^{n+1},\nabla\cdot Z_{\bu}^{n+1})\label{3.64}\\
		&=(F_{\xi}^{n+1},\nabla \cdot Z_{\bu}^{n+1}).\no
	\end{align}
	Using \reff{3.64}, we get 
	\begin{align}
		&(\mathcal{N}(\nabla\mathcal{R}_{h}\bu(t_{n+1})))-\mathcal{N}(\nabla\mathbf{u}_{h}^{n+1}),\varepsilon(Z_{\bu}^{n+1}))=(G_{\xi}^{n+1},\nabla\cdot Z_{\bu}^{n+1})\label{3.58}\\	&~~~~+(\mathcal{N}(\nabla\mathcal{R}_{h}\bu(t_{n+1}))-\mathcal{N}(\nabla\mathbf{u}(t_{n+1})),\varepsilon(Z_{\bu}^{n+1}))+(F_{\xi}^{n+1},\nabla\cdot Z_{\bu}^{n+1}).\nonumber
	\end{align}
	Combining \reff{3.207} and \reff{3.58}, we have
	\begin{align}
		&(\mathcal{N}(\nabla\mathcal{R}_{h}\bu(t_{n+1}))-\mathcal{N}(\nabla\mathbf{u}_{h}^{n+1}),\varepsilon(Z_{\bu}^{n+1}))\label{3.59}\\
		&=(\mathcal{N}(\nabla\mathcal{R}_{h}\bu(t_{n+1})))-\mathcal{N}(\nabla\mathbf{u}_{h}^{n+1}),\varepsilon(\mathcal{R}_{h}\bu(t_{n+1})-\varepsilon(\bu_{h}^{n+1}))\nonumber\\
		&\geq C_{4 }\left\|\varepsilon(\mathcal{R}_{h}\bu(t_{n+1})-\varepsilon(\bu_{h}^{n+1}) \right\|_{L^{2}(\varOmega)}^{2}=C_{4}\left\|\varepsilon(Z_{\bu}^{n+1}) \right\|_{L^{2}(\varOmega)}^{2}.\nonumber
	\end{align}
Using \reff{3.58} and \reff{3.59},	we have
	\begin{align}
		C_{4}\left\|\varepsilon(Z_{\bu}^{n+1}) \right\|_{L^{2}(\varOmega)}^{2}&\leq(\mathcal{N}(\nabla\mathcal{R}_{h}\bu(t_{n+1}))-\mathcal{N}(\nabla\mathbf{u}(t_{n+1})),\varepsilon(Z_{\bu}^{n+1}))\label{3.60}\\
		&+(G_{\xi}^{n+1},\nabla\cdot Z_{\bu}^{n+1})+(F_{\xi}^{n+1},\nabla\cdot Z_{\bu}^{n+1}).\nonumber
	\end{align}
	Setting $ \varphi_{h}=G_{\xi}^{n+1} $ after applying the difference operator $ d_{t} $ to \reff{3.54}, we get
	\begin{align}
		&\kappa_{3}(d_{t}G_{\xi}^{n+1},G_{\xi}^{n+1})+(\nabla\cdot (d_{t}Z_{\bu}^{n+1}),G_{\xi}^{n+1})=\kappa_{1}(d_{t}G_{\eta}^{n+\theta},G_{\xi}^{n+1})\label{3.61}\\
		&~~-(\nabla\cdot (d_{t}Y_{\bu}^{n+1}),G_{\xi}^{n+1})+\kappa_{1}(1-\theta)\varDelta t(d_{t}^{2}\eta(t_{n+1}),G_{\xi}^{n+1}).\nonumber
	\end{align}
	Setting $ \psi_{h}=\hat{Z}_{p}^{n+1}=F_{p}^{n+1}-Y_{p}^{n+1}+\kappa_{1}G_{\xi}^{n+1}+\kappa_{2}G_{\eta}^{n+\theta} $, we obtain
	\begin{align}
		&(d_{t}G_{\eta}^{n+\theta},\hat{Z}_{p}^{n+1})-(1-\theta)\dfrac{\kappa_{1}\varDelta t}{\mu_{f}}\left(Kd_{t} \nabla Z_{\xi}^{n+1},\nabla\hat{Z}_{p}^{n+1}  \right)\label{3.62} \\
		&+\dfrac{1}{\mu_{f}}(K(\nabla\hat{Z}_{p}^{n+1},\nabla\hat{Z}_{p}^{n+1})=(R_{h}^{n+\theta},\hat{Z}_{p}^{n+1})\no,
	\end{align}
	adding \reff{3.60}, \reff{3.61} and \reff{3.62},  applying the summation operator $ \varDelta t \sum_{n=1}^{l} $ to both sides, we get \reff{3.48}. The proof is complete.
\end{proof}
\begin{theorem}\label{2020}
	Let $ \left\lbrace (\bu_{h}^{n},\xi_{h}^{n},\eta_{h}^{n})\right\rbrace_{n\geq0}  $ be defined by the (MFEA), then there holds
	\begin{align}
		&\max_{0\leq n \leq l}\left[\sqrt{C_{4}}\left\|\varepsilon(Z_{\bu}^{n+1}) \right\|_{L^{2}(\varOmega)}+\sqrt{\kappa_{2}}\left\|G_{\eta}^{n+\theta} \right\|_{L^{2}(\varOmega)}+\sqrt{\kappa_{3}}\left\|G_{\xi}^{n+1} \right\|_{L^{2}(\varOmega)}\right]\label{3.63} \\
		&+\left[  \varDelta t\sum_{n=1}^{l}\dfrac{K}{\mu_{f}}\left\|\nabla\hat{Z}_{p}^{n+1} \right\|_{L^{2}(\varOmega)}^{2}  \right]\leq \hat{C}_{1}(T)\varDelta t+\hat{C}_{2}(T)h^{2} \nonumber
	\end{align}
	provided that $ \varDelta t=O(h^{2}) $ when $ \theta=0 $ and $ \varDelta t>0 $ when $ \theta=1 $. Here
	\begin{align}
		&\hat{C}_{1}(T)=\hat{C}\left\|\eta_{t} \right\|_{L^{2}((0,T);L^{2}(\varOmega))}+\hat{C}\left\|\eta_{tt}\right\|_{L^{2}((0,T);H^{1}(\varOmega)^{'})},\\
		& \hat C_{2}(T)=\hat{C}\left\|\xi_{t} \right\|_{L^{2}((0,T);H^{2}(\varOmega))}+\hat{C}\left\|\xi \right\|_{L^{\infty}((0,T);H^{2}(\varOmega))}\\
		&+\hat{C}\left\|\bu \right\|_{L^{2}((0,T);H^{3}(\varOmega))}+\hat{C}\left\|\nabla\cdot\bu_{t} \right\|_{L^{2}((0,T);H^{2}(\varOmega))}.\no
	\end{align}
\end{theorem}
\begin{proof}
Using \reff{3.48} and the fact of $ Z_{\boldsymbol{u}}^{0}=\boldsymbol{0}, Z_{\xi}^{0}=0$ and $ Z_{\eta}^{-1}=0$, we have
	\begin{align}
		&\mathcal{E}_{h}^{l}+\varDelta t\sum_{n=1}^{l}C_{4}\left\|\varepsilon(Z_{\bu}^{n+1}) \right\|_{L^{2}(\varOmega)}^{2}+\varDelta t\sum_{n=1}^{l}\left[\dfrac{K}{\mu_{f}}(\nabla\hat{Z}_{p}^{n+1},\nabla\hat{Z}_{p}^{n+1})\right.\label{3.66}\\
		&\left.+\dfrac{\kappa_{2}\varDelta t}{2}\left\|d_{t}G_{\eta}^{n+\theta}\right\|_{L^{2}(\varOmega)}^{2}+\dfrac{\kappa_{3}\varDelta t}{2}\left\|d_{t}G_{\xi}^{n+1}\right\|_{L^{2}(\varOmega)}^{2}\right] \nonumber\\
		&\leq \Phi_{1}+\Phi_{2}+\Phi_{3}+\Phi_{4}+\Phi_{5}+\Phi_{6},\nonumber
	\end{align}
where
	\begin{align*}
		&\Phi_{1}=\varDelta t\sum_{n=1}^{l}\left[
		(F_{\xi}^{n+1},\nabla\cdot Z_{\bu}^{n+1})-(\nabla\cdot d_{t}Z_{\bu}^{n+1},G_{\xi}^{n+1})\right],\\
		&\Phi_{2}=\varDelta t\sum_{n=1}^{l}\left[(G_{\xi}^{n+1},\nabla\cdot Z_{\bu}^{n+1})-(\nabla\cdot d_{t}Y_{\bu}^{n+1},G_{\xi}^{n+1})\right],\\
		&\Phi_{3}=\kappa_{1}(1-\theta)(\varDelta t)^{2}\sum_{n=1}^{l}(d_{t}^{2}\eta(t_{n+1}),G_{\xi}^{n+1}),\\
		&\Phi_{4}=\varDelta t\sum_{n=1}^{l}(R_{h}^{n+\theta},\hat{Z}_{p}^{n+1}),\\
		&\Phi_{5}=\varDelta t\sum_{n=1}^{l}(\mathcal{N}(\nabla\mathcal{R}_{h}\bu(t_{n+1}))-\mathcal{N}(\nabla\mathbf{u}(t_{n+1})),\varepsilon(Z_{\bu}^{n+1})), \\
		&\Phi_{6}=(1-\theta)(\varDelta t)^{2}\sum_{n=1}^{l}\dfrac{\kappa_{1}}{\mu_{f}}(Kd_{t}\nabla Z_{\xi}^{n+1},\nabla\hat{Z}_{p}^{n+1}),\\
		&\Phi_{7}=\varDelta t\sum_{n=1}^{l}(d_{t}G_{\eta}^{n+\theta},Y_{p}^{n+1}-F_{p}^{n+1}).
	\end{align*}
Next, we estimate each term on the right-hand of \reff{3.66}. For $\Phi_{1}$, using Korn's inequality, Cauchy-Schwarz inequality, Young inequality and the inequality of $ \left\|\nabla\cdot\bw\right\| _{L^2(\varOmega)}\leq c_{k}\left\| \varepsilon(\bw)\right\|_{L^2(\varOmega)}$ for all $\bw\in\bH_\perp^1(\Ome)$, we obtain
	\begin{align}\label{3.67}
		&\Phi_{1}=\varDelta t\sum_{n=1}^{l}\left[
		(F_{\xi}^{n+1},\nabla\cdot Z_{\bu}^{n+1})-(\nabla\cdot d_{t}Z_{\bu}^{n+1},G_{\xi}^{n+1})\right]\\
%		&\leq\dfrac{\varDelta t}{2}\sum_{n=1}^{l}\left[\left\|F_{\xi}^{n+1} \right\|_{L^{2}(\varOmega)}^{2}+\left\|\nabla\cdot Z_{\bu}^{n+1} \right\|_{L^{2}(\varOmega)}^{2}+\left\|\nabla\cdot (d_{t}Z_{\bu}^{n+1}) \right\|_{L^{2}(\varOmega)}^{2}\right.\no\\
%		&~~~~\left.+\left\|G_{\xi}^{n+1} \right\|_{L^{2}(\varOmega)}^{2}  \right] \nonumber\\
		&\leq\dfrac{\varDelta t}{2}\sum_{n=1}^{l}\left[\left\|F_{\xi}^{n+1} \right\|_{L^{2}(\varOmega)}^{2}+c_{k}\left\|\varepsilon(Z_{\bu}^{n+1}) \right\|_{L^{2}(\varOmega)}^{2}+c_{k}\left\|d_{t}\varepsilon(Z_{\bu}^{n+1})\right\|_{L^{2}(\varOmega)}^{2}\right.\no\\
		&\left.~~~~+\left\|G_{\xi}^{n+1} \right\|_{L^{2}(\varOmega)}^{2}\right].\nonumber
	\end{align}
	Similarly, using Korn's inequality, the Cauchy-Schwarz inequality and Young inequality for $ \Phi_{2}$, we get
	\begin{align}
		&\Phi_{2}=\varDelta t\sum_{n=1}^{l}\left[(G_{\xi}^{n+1},\nabla\cdot Z_{\bu}^{n+1})-(\nabla\cdot d_{t}Y_{\bu}^{n+1},G_{\xi}^{n+1})\right]\label{3.68} \\
%		&\leq\dfrac{\varDelta t}{2}\sum_{n=1}^{l}\left[\left\|G_{\xi}^{n+1} \right\|_{L^{2}(\varOmega)}^{2}+\left\|\nabla\cdot Z_{\bu}^{n+1} \right\|_{L^{2}(\varOmega)}^{2}+\left\|\nabla\cdot (d_{t}Y_{\bu}^{n+1}) \right\|_{L^{2}(\varOmega)}^{2}\right.\no\\
%		&~~~~\left.+\left\|G_{\xi}^{n+1} \right\|_{L^{2}(\varOmega)}^{2}\right]\nonumber\\
		&\leq\dfrac{\varDelta t}{2}\sum_{n=1}^{l}\left[\left\|G_{\xi}^{n+1} \right\|_{L^{2}(\varOmega)}^{2}+c_{k}\left\|\varepsilon(Z_{\bu}^{n+1} )\right\|_{L^{2}(\varOmega)}^{2}+\left\|\nabla\cdot (d_{t}Y_{\bu}^{n+1}) \right\|_{L^{2}(\varOmega)}^{2}\right.\no\\
		&~~~~\left.+\left\|G_{\xi}^{n+1} \right\|_{L^{2}(\varOmega)}^{2}\right].\nonumber
	\end{align} 
	When $ \theta=0$, using the integration by parts and $d_{t}\eta(t_{0})=0$, we get 
\begin{eqnarray}
		&&\Phi_{3}=\kappa_{1}(\varDelta t)^{2}\sum_{n=1}^{l}(d_{t}^{2}\eta(t_{n+1}),G_{\xi}^{n+1})\label{3.69}\\
		&&\quad=\kappa_{1}(\varDelta t)^{2}\left[ \dfrac{1}{\varDelta t}(d_{t}\eta(t_{l+1}),G_{\xi}^{l+1})-\sum_{n=1}^{l}(d_{t}\eta(t_{n+1}),d_{t}G_{\xi}^{n+1})\right].\no
\end{eqnarray}
	Using the Cauchy-Schwarz inequality, Young inequality and \reff{3.3}, we have
	\begin{align}
		&\dfrac{1}{\varDelta t}(d_{t}\eta(t_{l+1}),G_{\xi}^{l+1})\leq\dfrac{1}{\varDelta t}\left\|d_{t}\eta(t_{l+1}) \right\|_{L^{2}(\varOmega)}\left\|G_{\xi}^{l+1} \right\|_{L^{2}(\varOmega)}\label{3.70}\\
		&\leq\dfrac{1}{\varDelta t}\left\|\eta_{t} \right\|_{L^{2}((t_{l},t_{l+1});\varOmega)}\cdot\dfrac{1}{\beta_{1}}\sup_{\bv_{h}\in V_{h}}\left[ \dfrac{(\mathcal{N}(\nabla\bu(t_{l+1}))-\mathcal{N}(\nabla\mathbf{u}_{h}^{l+1}),\varepsilon(\bv_{h}))}{\left\|\nabla\bv_{h} \right\|_{L^{2}(\varOmega)}}\right.\nonumber\\
		&~~~~\left.-\dfrac{(F_{\xi}^{l+1},\nabla\cdot\bv_{h})}{\left\|\nabla\bv_{h} \right\|_{L^{2}(\varOmega)}}\right] \nonumber\\
		&\leq\dfrac{1}{\beta_{1}\varDelta t}\left\|\eta_{t} \right\|_{L^{2}((t_{l},t_{l+1});\varOmega)}\left[C_{3} \left\| \varepsilon(Z_{\bu}^{l+1})\right\|_{L^{2}(\varOmega)}+C_{3}\left\| \varepsilon(Y_{\bu}^{l+1})\right\|_{L^{2}(\varOmega)}\right.\no\\
		&~~~~\left.+c_{k}\left\| F_{\xi}^{l+1}\right\|_{L^{2}(\varOmega)} \right]\nonumber\\
		&\leq\dfrac{3}{\beta_{1}^{2}}\left\|\eta_{t} \right\|_{L^{2}((t_{l},t_{l+1});\varOmega)}^{2}+\dfrac{C_{3}^{2}}{4\varDelta t^{2}}\left\| \varepsilon(Z_{\bu}^{l+1})\right\|_{L^{2}(\varOmega)}^{2}+\dfrac{C_{3}^{2}}{4\varDelta t^{2}}\left\| \varepsilon(Y_{\bu}^{l+1})\right\|_{L^{2}(\varOmega)}^{2}\no\\
		&+\dfrac{c_{k}^{2}}{4\varDelta t^{2}}\left\| F_{\xi}^{l+1}\right\|_{L^{2}(\varOmega)}^{2},\nonumber
	\end{align}
	\begin{eqnarray}
		&\sum_{n=1}^{l}(d_{t}\eta(t_{n+1}),d_{t}G_{\xi}^{n+1})\leq\sum_{n=1}^{l}\left\| d_{t}\eta(t_{n+1})\right\|_{L^{2}(\varOmega)}\left\| d_{t}G_{\xi}^{n+1}\right\|_{L^{2}(\varOmega)}\label{3.71}\\
		&\leq\sum_{n=1}^{l}\left\| d_{t}\eta(t_{n+1})\right\|_{L^{2}(\varOmega)}\cdot\dfrac{1}{\beta_{1}}\sup_{\bv_{h}\in V_{h}}\left[ \dfrac{(d_{t}\mathcal{N}(\nabla\bu( t_{n+1}))-d_{t}\mathcal{N}(\nabla\mathbf{u}_{h}^{n+1}),\varepsilon(\bv_{h}))}{\left\|\nabla\bv_{h} \right\|_{L^{2}(\varOmega)}}\right.\nonumber\\
		&~~~~\left.-\dfrac{(d_{t}F_{\xi}^{n+1},\nabla\cdot\bv_{h})}{\left\|\nabla\bv_{h} \right\|_{L^{2}(\varOmega)}}\right] \nonumber\\
		&\leq\sum_{n=1}^{l}\dfrac{1}{\beta_{1}}\left\| d_{t}\eta(t_{n+1})\right\|_{L^{2}(\varOmega)}\left[C_{3} \left\| d_{t}\varepsilon(Z_{\bu}^{n+1})\right\|_{L^{2}(\varOmega)}+C_{3}\left\| d_{t}\varepsilon(Y_{\bu}^{n+1})\right\|_{L^{2}(\varOmega)}\right.\no\\
		&~~~~\left.+c_{k}\left\| d_{t}F_{\xi}^{n+1}\right\|_{L^{2}(\varOmega)} \right]\nonumber\\
		&\leq\dfrac{3}{\beta_{1}^{2}}\left\| \eta_{t}\right\|_{L^{2}((0,T);L^{2}(\varOmega))}^{2}+\sum_{n=1}^{l}\left[ \dfrac{C_{3}^{2}}{4}\left\| d_{t}\varepsilon(Z_{\bu}^{n+1})\right\|_{L^{2}(\varOmega)}^{2}+\dfrac{C_{3}^{2}}{4}\left\| d_{t}\varepsilon(Y_{\bu}^{n+1})\right\|_{L^{2}(\varOmega)}^{2}\right.\no\\
		&~~~~\left.+\dfrac{c_{k}^{2}}{4}\left\| d_{t}F_{\xi}^{n+1})\right\|_{L^{2}(\varOmega)}^{2}\right]. \nonumber
	\end{eqnarray}
The term of	$\Phi_{4}$ can be bounded by
	\begin{align}
		&\left| \varDelta t\sum_{n=1}^{l}(R_{h}^{n+\theta},\hat{Z}_{p}^{n+1})\right|\leq\varDelta t\sum_{n=1}^{l}\left\|R_{h}^{n+\theta} \right\|_{H^{1}(\varOmega)^{'}}\left\| \nabla\hat{Z}_{p}^{n+1}\right\|_{L^{2}(\varOmega)}\label{3.72}\\
		&\leq\varDelta t\sum_{n=1}^{l}\left[ \dfrac{K}{4\mu_{f}}\left\| \nabla\hat{Z}_{p}^{n+1}\right\|_{L^{2}(\varOmega)}^{2}+\dfrac{\mu_{f}}{K}\left\|R_{h}^{n+\theta} \right\|_{H^{1}(\varOmega)^{'}}^{2}\right] \nonumber\\
		&\leq\varDelta t\sum_{n=1}^{l}\left[ \dfrac{K}{4\mu_{f}}\left\| \nabla\hat{Z}_{p}^{n+1}\right\|_{L^{2}(\varOmega)}^{2}+\dfrac{\mu_{f}\varDelta t}{3K}\left\|\eta_{tt} \right\|_{L^{2}((t_{n},t_{n+1});H^{1}(\varOmega)^{'})}^{2} \right],\nonumber
	\end{align}
	where we used the fact that
	\begin{align*}
		\left\|R_{h}^{n+\theta} \right\|_{H^{1}(\varOmega)^{'}}^{2}\leq\dfrac{\varDelta  t}{3}\int_{t_{n}}^{t_{n+1}}\left\|\eta_{tt} \right\|_{H^{1}(\varOmega)^{'}}^{2}dt. 
	\end{align*}
	
	As for the term of $ \Phi_{5}$, using the Cauchy-Schwarz inequality, Young inequality and \reff{3.206}, we have
	\begin{align}
		&\varDelta t\sum_{n=1}^{l}(\mathcal{N}(\nabla\mathcal{R}_{h}\bu(t_{n+1}))-\mathcal{N}(\nabla(\bu(t_{n+1})),\varepsilon(Z_{\bu}^{n+1}))\label{3.73}\\
		&\leq\varDelta t\sum_{n=1}^{l}\left\| \mathcal{N}(\nabla\mathcal{R}_{h}\bu(t_{n+1}))-\mathcal{N}((\nabla\bu(t_{n+1}))\right\|_{L^{2}(\varOmega)}\left\| \varepsilon(Z_{\bu}^{n+1})\right\|_{L^{2}(\varOmega)}\nonumber\\
		&\leq\varDelta t\sum_{n=1}^{l}C_{3}\left\| \varepsilon(\mathcal{R}_{h}\boldsymbol{u}(t_{n+1}))-\varepsilon(\bu(t_{n+1})\right\|_{L^{2}(\varOmega)}\left\| \varepsilon(Z_{\bu}^{n+1})\right\|_{L^{2}(\varOmega)}\nonumber\\
		&\leq\varDelta t\sum_{n=1}^{l}\left[ \dfrac{C_{3}^{2}}{4}\left\| \varepsilon(\mathcal{R}_{h}\bu(t_{n+1}))-\varepsilon(\bu(t_{n+1})\right\|_{L^{2}(\varOmega)}^{2}+\left\| \varepsilon(Z_{\bu}^{n+1})\right\|_{L^{2}(\varOmega)}^{2}\right]\nonumber\\
		&=\varDelta t\sum_{n=1}^{l}\left[ \dfrac{C_{3}^{2}}{4}\left\| \varepsilon(Y_{\bu}^{n+1})\right\|_{L^{2}(\varOmega)}^{2}+\left\| \varepsilon(Z_{\bu}^{n+1})\right\|_{L^{2}(\varOmega)}^{2}\right].\nonumber
	\end{align}
As for $\Phi_{6}$, using the inverse inequality \reff{20210915}, Cauchy-Schwarz inequality, Young inequality and inf-sup condition, we have 
	\begin{align}
		&\Phi_{6}=(\varDelta t)^{2}\sum_{n=1}^{l}\dfrac{\kappa_{1}}{\mu_{f}}(Kd_{t}\nabla Z_{\xi}^{n+1},\nabla\hat{Z}_{p}^{n+1})\label{3.74}\\
		&\leq(\varDelta t)^{2}\sum_{n=1}^{l}\dfrac{c_{1}h^{-1}K\kappa_{1}}{\mu_{f}}\left\| d_{t}Z_{\xi}^{n+1}\right\|_{L^{2}(\varOmega)}\left\| \nabla\hat{Z}_{p}^{n+1}\right\|_{L^{2}(\varOmega)}\nonumber\\
		&\leq(\varDelta t)^{2}\sum_{n=1}^{l}\dfrac{c_{1}h^{-1}K\kappa_{1}}{\mu_{f}\beta_{1}}\sup_{\bv_{h}\in V_{h}}\left[ \dfrac{(d_{t}\mathcal{N}(\bu(t_{n+1}))-d_{t}\mathcal{N}(\nabla\mathbf{u}_{h}^{n+1}),\varepsilon(\bv_{h}))}{\left\|\nabla\bv_{h} \right\|_{L^{2}(\varOmega)}}\right.\no\\
		&~~~~\left.-\dfrac{(d_{t}Y_{\xi}^{n+1},\nabla\cdot\bv_{h})}{\left\|\nabla\bv_{h} \right\|_{L^{2}(\varOmega)}}\right]\left\| \nabla\hat{Z}_{p}^{n+1}\right\|_{L^{2}(\varOmega)}\nonumber\\
		&\leq(\varDelta t)^{2}\dfrac{K}{\mu_{f}}\sum_{n=1}^{l}\dfrac{\kappa_{1}c_{1}}{h\beta_{1}}\left[C_{3}\left\|d_{t}\varepsilon(Z_{\bu}^{n+1})\right\|_{L^{2}(\varOmega)}+C_{3}\left\| d_{t}\varepsilon(Y_{\bu}^{n+1})\right\|_{L^{2}(\varOmega)}\right.\nonumber\\
		&~~~~\left.+c_{k}\left\| d_{t}Y_{\xi}^{n+1}\right\|_{L^{2}(\varOmega)} \right]\left\| \nabla\hat{Z}_{p}^{n+1}\right\|_{L^{2}(\varOmega)} \nonumber\\
		&\leq(\varDelta t)^{2}\dfrac{K}{\mu_{f}}\sum_{n=1}^{l}\left[\dfrac{C_{3}^{2}\kappa_{1}^{2}c_{1}^{2}\varDelta t}{h^{2}\beta_{1}^{2}}\left\|d_{t}\varepsilon(Z_{\bu}^{n+1})\right\|_{L^{2}(\varOmega)}^{2}\right.\nonumber\\
		&~~~~+\dfrac{C_{3}^{2}\kappa_{1}^{2}c_{1}^{2}\varDelta t}{h^{2}\beta_{1}^{2}}\left\| d_{t}\varepsilon(Y_{\bu}^{n+1})\right\|_{L^{2}(\varOmega)}^{2}+\dfrac{\varDelta t\kappa_{1}^{2}c_{1}^{2}c_{k}^{2}}{h^{2}\beta_{1}^{2}}\left\| d_{t}Y_{\xi}^{n+1}\right\|_{L^{2}(\varOmega)}^{2}\no\\
		&~~~~\left.+\dfrac{3}{4\varDelta t}\left\| \nabla\hat{Z}_{p}^{n+1}\right\|_{L^{2}(\varOmega)}^{2}\right].\nonumber
	\end{align}
	Using the Cauchy-Schwarz inequality and Young inequality, we get
	\begin{align}
		&\Phi_{7}=\varDelta t\sum_{n=1}^{l}(d_{t}G_{\eta}^{n+\theta},Y_{p}^{n+1}-F_{p}^{n+1})\label{3.75}\\
		&\leq \varDelta t\sum_{n=1}^{l}\left\| d_{t}G_{\eta}^{n+\theta}\right\|_{L^{2}(\varOmega)}\left\| Y_{p}^{n+1}-F_{p}^{n+1}\right\|_{L^{2}(\varOmega)}\nonumber\\
		&\leq\varDelta t\sum_{n=1}^{l}\left\| d_{t}G_{\eta}^{n+\theta}\right\|_{L^{2}(\varOmega)}\left(\left\| F_{p}^{n+1}\right\|_{L^{2}(\varOmega)}+\left\| Y_{p}^{n+1}\right\|_{L^{2}(\varOmega)}\right)\nonumber\\
		& \leq\varDelta t\sum_{n=1}^{l}\left[2\left\| d_{t}G_{\eta}^{n+\theta}\right\|_{L^{2}(\varOmega)}^{2}+\dfrac{1}{4}\left\| F_{p}^{n+1}\right\|_{L^{2}(\varOmega)}^{2}+\dfrac{1}{4}\left\| Y_{p}^{n+1}\right\|_{L^{2}(\varOmega)}^{2}\right]. \nonumber
	\end{align}
	Substituting \reff{3.67}-\reff{3.75}  into \reff{3.66}, we have
	\begin{eqnarray}
		&&\qquad\dfrac{1}{2}\left[\kappa_{2}\left\|G_{\eta}^{l+\theta} \right\|_{L^{2}(\varOmega)}^{2}+\kappa_{3}\left\|G_{\xi}^{l+1} \right\|_{L^{2}(\varOmega)}^{2}  \right]\label{202193}\\
		&&~~+\varDelta t\sum_{n=1}^{l}C_{4}\left\|\varepsilon(Z_{\bu}^{n+1}) \right\|_{L^{2}(\varOmega)}^{2}+\varDelta t\sum_{n=1}^{l}\left[\dfrac{K}{\mu_{f}}(\nabla\hat{Z}_{p}^{n+1},\nabla\hat{Z}_{p}^{n+1})\right.\no\\
		&&~~\left.+\dfrac{\kappa_{2}\varDelta t}{2}\left\|d_{t}G_{\eta}^{n+\theta}\right\|_{L^{2}(\varOmega)}^{2}+\dfrac{\kappa_{3}\varDelta t}{2}\left\|d_{t}G_{\xi}^{n+1}\right\|_{L^{2}(\varOmega)}^{2}\right]\no\\
		&&\leq\dfrac{\varDelta t}{2}\sum_{n=1}^{l}\left[\left\|F_{\xi}^{n+1} \right\|_{L^{2}(\varOmega)}^{2}+c_{k}\left\|\varepsilon(Z_{\bu}^{n+1}) \right\|_{L^{2}(\varOmega)}^{2}+c_{k}\left\|d_{t}\varepsilon(Z_{\bu}^{n+1})\right\|_{L^{2}(\varOmega)}^{2}\right.\no\\
		&&~~\left.+\left\|G_{\xi}^{n+1} \right\|_{L^{2}(\varOmega)}^{2}\right]+\dfrac{\varDelta t}{2}\sum_{n=1}^{l}\left[\left\|G_{\xi}^{n+1} \right\|_{L^{2}(\varOmega)}^{2}+c_{k}\left\|\varepsilon(Z_{\bu}^{n+1} )\right\|_{L^{2}(\varOmega)}^{2}\right.\nonumber\\
		&&~~\left.+\left\|\nabla\cdot (d_{t}Y_{\bu}^{n+1}) \right\|_{L^{2}(\varOmega)}^{2}+\left\|G_{\xi}^{n+1} \right\|_{L^{2}(\varOmega)}^{2}\right]+\dfrac{3k_{1}(\varDelta t)^{2}}{\beta_{1}^{2}}\left\|\eta_{t} \right\|_{L^{2}((t_{l},t_{l+1});\varOmega)}^{2}\nonumber\\
		&&~~+\dfrac{C_{3}^{2}}{4}\left\| \varepsilon(Z_{\bu}^{l+1})\right\|_{L^{2}(\varOmega)}^{2}+\dfrac{C_{3}^{2}}{4}\left\| \varepsilon(Y_{\bu}^{l+1})\right\|_{L^{2}(\varOmega)}^{2}+\dfrac{c_{k}^{2}}{4}\left\| F_{\xi}^{l+1}\right\|_{L^{2}(\varOmega)}^{2}\no\\
		&&~~+\dfrac{3k_{1}(\varDelta t)^{2}}{\beta_{1}^{2}}\left\| \eta_{t}\right\|_{L^{2}((0,T);L^{2}(\varOmega))}^{2}+k_{1}(\varDelta t)^{2}\sum_{n=1}^{l}\left[ \dfrac{C_{3}^{2}}{4}\left\| d_{t}\varepsilon(Z_{\bu}^{n+1})\right\|_{L^{2}(\varOmega)}^{2}\right.\no\\
		&&~~\left.+\dfrac{C_{3}^{2}}{4}\left\| d_{t}\varepsilon(Y_{\bu}^{n+1})\right\|_{L^{2}(\varOmega)}^{2}+\dfrac{c_{k}^{2}}{4}\left\| d_{t}F_{\xi}^{n+1})\right\|_{L^{2}(\varOmega)}^{2}\right] \nonumber\\
		&&~~+\varDelta t\sum_{n=1}^{l}\left[ \dfrac{K}{4\mu_{f}}\left\| \nabla\hat{Z}_{p}^{n+1}\right\|_{L^{2}(\varOmega)}^{2}+\dfrac{\mu_{f}\varDelta t}{3K}\left\|\eta_{tt} \right\|_{L^{2}((t_{n},t_{n+1});H^{1}(\varOmega)^{'})}^{2} \right]\nonumber\\
		&&~~+\varDelta t\sum_{n=1}^{l}\left[ \dfrac{C_{3}^{2}}{4}\left\| \varepsilon(Y_{\bu}^{n+1})\right\|_{L^{2}(\varOmega)}^{2}+\left\| \varepsilon(Z_{\bu}^{n+1})\right\|_{L^{2}(\varOmega)}^{2}\right]\nonumber\\
		&&~~+(\varDelta t)^{2}\dfrac{K}{\mu_{f}}\sum_{n=1}^{l}\left[\dfrac{C_{3}^{2}\kappa_{1}^{2}c_{1}^{2}\varDelta t}{h^{2}\beta_{1}^{2}}\left\|d_{t}\varepsilon(Z_{\bu}^{n+1})\right\|_{L^{2}(\varOmega)}^{2}+\dfrac{3}{4\varDelta t}\left\| \nabla\hat{Z}_{p}^{n+1}\right\|_{L^{2}(\varOmega)}^{2}\right.\nonumber\\
		&&~~\left.+\dfrac{C_{3}^{2}\kappa_{1}^{2}c_{1}^{2}\varDelta t}{h^{2}\beta_{1}^{2}}\left\| d_{t}\varepsilon(Y_{\bu}^{n+1})\right\|_{L^{2}(\varOmega)}^{2}+\dfrac{\varDelta t\kappa_{1}^{2}c_{1}^{2}c_{k}^{2}}{h^{2}\beta_{1}^{2}}\left\| d_{t}Y_{\xi}^{n+1}\right\|_{L^{2}(\varOmega)}^{2}\right]\nonumber\\
		&& ~~+\varDelta t\sum_{n=1}^{l}\left[2\left\| d_{t}G_{\eta}^{n+\theta}\right\|_{L^{2}(\varOmega)}^{2}+\dfrac{1}{4}\left\| F_{p}^{n+1}\right\|_{L^{2}(\varOmega)}^{2}+\dfrac{1}{4}\left\| Y_{p}^{n+1}\right\|_{L^{2}(\varOmega)}^{2}\right]. \nonumber
	\end{eqnarray}
  Applying the discrete Gronwall inequality (cf. \cite{Shen1990}) to \reff{202193}, we obtain
	\begin{eqnarray}
		&&\quad C_{4}\left\|\varepsilon(Z_{\bu}^{l+1}) \right\|_{L^{2}(\varOmega)}^{2}+\kappa_{2}\left\|G_{\eta}^{l+\theta}\right\|_{L^{2}(\varOmega)}^{2}+\kappa_{3}\left\|G_{\xi}^{l+1} \right\|_{L^{2}(\varOmega)}^{2}\\
		&&~~+\varDelta t\sum_{n=1}^{l}\dfrac{K}{\mu_{f}}\left\|\nabla \hat{Z}_{p}^{n+1} \right\|_{L^{2}(\varOmega)}^{2}\no\\
		&&\leq\dfrac{\mu_{f}(\varDelta t)^{2}}{3K}\left\|\eta_{tt}\right\|_{L^{2}((0,T);H^{1}(\varOmega)^{'})}^{2}+\dfrac{6k_{1}(\varDelta t)^{2}}{\beta_{1}^{2}}\left\|\eta_{t} \right\|_{L^{2}((0,T);L^{2}(\varOmega))}^{2}\nonumber\\
		&&~~~~+\hat{C_{1}}\left\|F_{\xi}^{l+1}\right\|_{L^{2}(\varOmega)}^{2}+\varDelta t\sum_{n=1}^{l}\left\|d_{t}F_{\xi}^{n+1}\right\|_{L^{2}(\varOmega)}^{2}+\varDelta t\sum_{n=1}^{l}\left\|\nabla\cdot d_{t}Y_{\bu}^{n+1}\right\|_{L^{2}(\varOmega)}^{2}\nonumber\\
		&&~~+C_{3}\hat{C}_{1}\left\|\varepsilon(Y_{\bu}^{l+1})\right\|_{L^{2}(\varOmega)}^{2}+\varDelta t\sum_{n=1}^{l}\left\|F_{p}^{n+1}\right\|_{L^{2}(\varOmega)}^{2}+\varDelta t\sum_{n=1}^{l}\left\|Y_{p}^{n+1}\right\|_{L^{2}(\varOmega)}^{2}
		\nonumber\\
		&&\leq\hat{C}(\varDelta t)^{2}\left( \left\|\eta_{t} \right\|_{L^{2}((0,T);L^{2}(\varOmega))}^{2}+\left\|\eta_{tt}\right\|_{L^{2}((0,T);H^{1}(\varOmega)^{'})}^{2}\right) \nonumber\\
		&&~~+\hat{C}h^{4}\left[\left\|\xi_{t} \right\|_{L^{2}((0,T);H^{2}(\varOmega))}^{2}+\left\|\xi \right\|_{L^{\infty}((0,T);H^{2}(\varOmega))}^{2}+\left\|\bu \right\|_{L^{2}((0,T);H^{3}(\varOmega))}^{2}\right.\no\\
		&&~~\left.+\left\|\nabla\cdot\bu_{t} \right\|_{L^{2}((0,T);H^{2}(\varOmega))}^{2}\right] \nonumber 
	\end{eqnarray}
	provided that $ \varDelta t\leq\dfrac{h^{2}\beta_{1}^{2}\mu_{f}C_{4}}{4K\kappa_{1}^{2}c_{1}^{2}C_{3}^{2}} $ when $ \theta=0 $ or $ \varDelta t \geq0 $ when $\theta=1$. Hence, we  deduce that \reff{3.63} holds. The proof is complete.
\end{proof}
\begin{theorem}
	The solution of the (MFEA) satisfies the following error estimates:
	\begin{align}
		&\max_{0\leq n\leq N}\left[\sqrt{C_{4}}\left\|\nabla(\bu(t_{n+1})-\bu_{h}^{n+1})\right\|_{L^{2}(\varOmega )}+\sqrt\kappa_{2}\left\|\eta(t_{n+1})-\eta_{h}^{n+1}\right\|_{L^{2}(\varOmega )}\right.\\
		&\left.+\sqrt{\kappa_{3}}\left\|\xi(t_{n+1})-\xi_{h}^{n+1}\right\|_{L^{2}(\varOmega )} \right]\leq\check{C}_{1}(T)\varDelta t+\check{C}_{2}(T)h^{2},\nonumber\\
		&\left(\varDelta t\sum_{n=0}^{N}\dfrac{K}{\mu_{f}}\left\|\nabla p(t_{n+1})-\nabla p_{h}^{n+1}\right\|_{L^{2}(\varOmega )}^{2} \right)^{\frac{1}{2}}\leq \check{C}_{1}(T)\varDelta t+\check{C}_{2}(T)h
	\end{align}
	provided that $ \varDelta t=O(h^{2}) $ when $ \theta=0 $ and $ \varDelta t\geq0 $ when $ \theta=1$. Here
	$\check{C}_{1}(T)=\hat{C}_{1}(T)$, $
		\check{C}_{2}(T)=\hat{C}_{2}(T)+\left\|\xi \right\|_{L^{\infty}((0,T);H^{2}(\varOmega))}+\left\|\eta \right\|_{L^{\infty}((0,T);H^{2}(\varOmega))}+\left\|\nabla\boldsymbol{u} \right\|_{L^{\infty}((0,T);H^{2}(\varOmega))}.
	$
\end{theorem}
\begin{proof}
	The above estimates follow immediately from an application of the triangle inequality on
	\begin{align*}
		&\bu(t_{n})-\bu_{h}^{n}=Y_{\bu}^{n}+Z_{\bu}^{n},~~~~~~~~~~~~~~~~~~~~~\xi(t_{n})-\xi_{h}^{n}=Y_{\xi}^{n}+Z_{\xi}^{n}=F_{\xi}^{n}+G_{\xi}^{n},\\
		&\eta(t_{n})-\eta_{h}^{n}=Y_{\eta}^{n}+Z_{\eta}^{n}=F_{\eta}^{n}+G_{\eta}^{n},~~~~~~~p(t_{n})-p_{h}^{n}=Y_{p}^{n}+Z_{p}^{n}=F_{p}^{n}+G_{p}^{n}.
	\end{align*}
	and appealing to \reff{3.41}, \reff{3.44}, \reff{3.46} and Theorem \ref{2020}. The proof is complete.
\end{proof} 

\section{Numerical tests}\label{sec-4}
%In this section we shall present two 2-dimensional
%numerical experiments to validate theoretical results.
~

\medskip
{\bf Test 1.} Let $\Omega=[0,1]\times [0,1]$, $\Gamma_1=\{(1,x_2); 0\leq x_2\leq 1\}$,
$\Gamma_2=\{(x_1,0); 0\leq x_1\leq 1\}$, $\Gamma_3=\{(0,x_2); 0\leq x_2\leq 1\}$,
$\Gamma_4=\{(x_1,1); 0\leq x_1\leq 1\}$, and $T=1$. The source functions are as follows:
\begin{align*}
	\mathbf{f} &=-(\lambda+\mu) t(1,1)^T-2(\mu+\lambda) t^{2}(x_{1},x_{2})^{T}+\alpha te^{x_1+x_2}(1,1)^T,\\
	\phi &=c_0e^{x_1+x_2}-\frac{2K}{\mu_f}te^{x_1+x_2}+\alpha(x_1+x_2),
\end{align*}
and the boundary and initial conditions are 
\begin{alignat*}{2}
	p &= te^{x_1+x_2}  &&\qquad\mbox{on }\partial\Omega_T,\\
	u_1 &= \frac12 x_1^2t &&\qquad\mbox{on }\Gamma_j\times (0,T),\, j=1,3,\\
	u_2 &= \frac12 x_2^2t &&\qquad\mbox{on }\Gamma_j\times (0,T),\, j=,2,4,\\
	\sigma\bf{n}-\alpha \emph{p}\bf{n} &= \mathbf{f}_1 &&\qquad \mbox{on } \p\Ome_T,\\
	\mathbf{u}(x,0) = \mathbf{0},  \quad p(x,0) &=0 &&\qquad\mbox{in } \Ome,
\end{alignat*}
where
\begin{align*}
	\mathbf{f}_1(x,t)= \lambda(x_1+x_2) (n_1,n_2)^T t +\mu t(x_{1}n_{1},x_{2}n_{2})^{T}+\mu t^{2}\bigl(x_{1}^{2}n_{1},x_{2}^{2}n_{2}\bigr)^{T}\\
	+\lambda t^{2}(x_{1}^{2}+x_{2}^{2})(n_1,n_2)^T-\alpha(n_1,n_2)^Tt e^{x_1+x_2}.
\end{align*}
The exact solution of this problem is
\[
\mathbf{u}(x,t)=\frac{t}2 \bigl( x_1^2, x_2^2 \bigr)^T,\quad p(x,t)=te^{x_1+x_2}.
\]
\begin{table}[H]
	\begin{center}
		%\smallskip
		\caption{Values of parameters} \label{tab101}
		\begin{tabular}{l c c }
			\hline
			Parameters &Description  &Values  \\ \hline
			$\nu $ &  Poisson ratio   & 0.25      \\ %\hline
			$\alpha$ &  Biot-Willis constant   & 1e-5   \\ %\hline
			$E$ &  Young's modulus   & 0.25 \\ %\hline
			$\lambda$ &  Lam$\acute{e}$ constant   & 0.1  \\%\hline
			$ K $ & Permeability tensor  & (1e-3) $\bf I$ \\% \hline
			$ \mu $ &  Lam$\acute{e}$ constant  & 0.1\\%\hline
			$ c_{0} $ &  Constrained specific storage coefficient  & 2\\\hline	
		\end{tabular}
	\end{center}
\end{table}
\begin{table}[!htbp]
	\begin{center}
		\caption{Spatial errors and convergence rates of $\mathbf{u}$} \label{tab1}
		\begin{tabular}{l c c c c }
			\hline
			$h$ & $\|\mathbf{u}-\mathbf{u}_h\|_{L^2}$ &CR &$ \|\mathbf{u}-\mathbf{u}_h\|_{H^1}$ & CR\\ \hline
			$h=1/3$ & 1.5015e-6  &    & 1.9107e-5       &     \\ %\hline
			$h=1/6$ & 2.3515e-7  & 2.675  & 3.3561e-6     &  2.509     \\ %\hline
			$h=1/12$ & 6.3182e-8  & 1.896  & 7.2141e-7      & 2.218      \\ %\hline
			$h=1/24$ & 1.4724e-8  & 2.101   & 1.976e-7     &  1.868      \\ \hline
		\end{tabular}
	\end{center}
\end{table}
\begin{table}[!htbp]
	\begin{center}
		\caption{Spatial errors and convergence rates of $p$} \label{tab106}
		\begin{tabular}{l c c c c }
			\hline
			$h$ & $\|{p}-{p}_h\|_{L^2}$ &CR &$ \|{p}-{p}_h\|_{H^1}$ & CR\\ \hline
			$h=1/3$ & 0.032683 &    & 0.43705      &     \\ %\hline
			$h=1/6$ & 0.008109  & 2.011  & 0.21766    &  1.005     \\ %\hline
			$h=1/12$ & 0.002016  & 2.008  & 0.10872      &  1.002     \\ %\hline
			$h=1/24$ & 0.000497  & 2.022   & 0.05435     &  1.0004      \\ \hline
		\end{tabular}
	\end{center}
\end{table}
\begin{figure}[H]
	\centering
	\includegraphics[height=2.2in,width=3in]{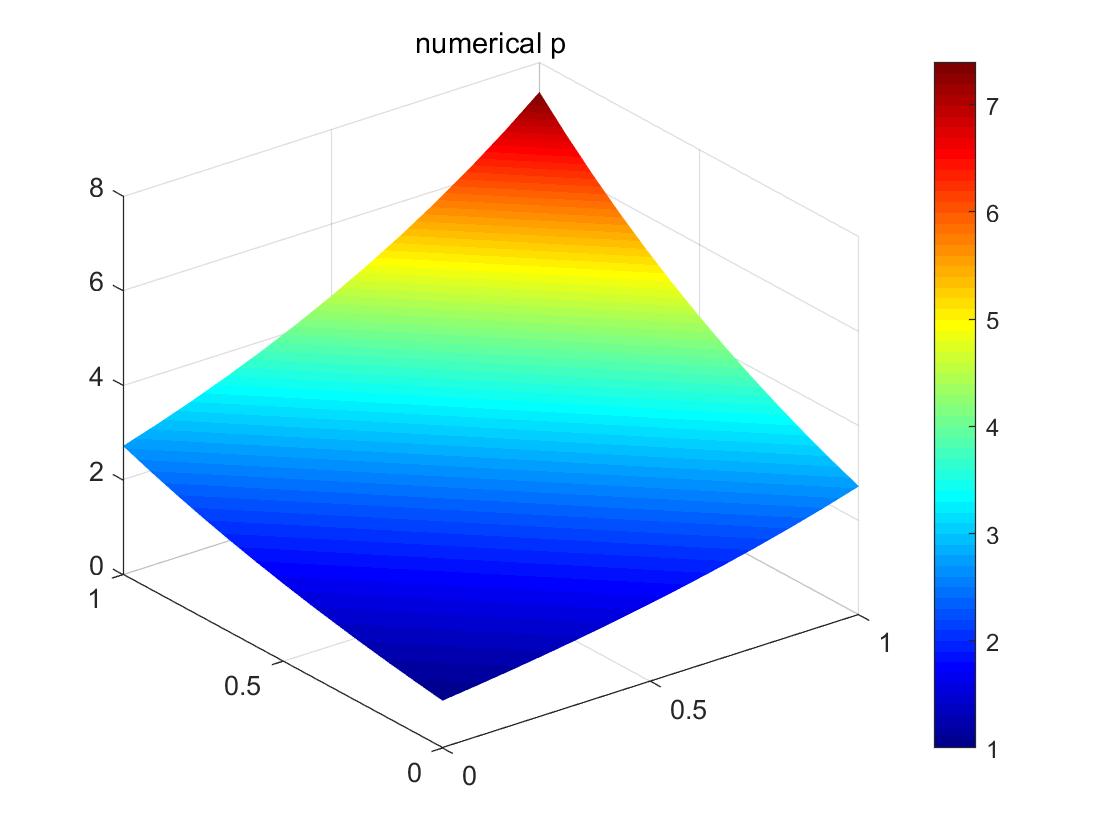}
	\caption{The numerical  pressure\ $p_{h}^{n+1}$ at the terminal time $T$ with the parameters of Table \ref{tab101}.}\label{figure_p1}
\end{figure}
\begin{figure}[H]
	\centering
	\includegraphics[height=2.2in,width=3in]{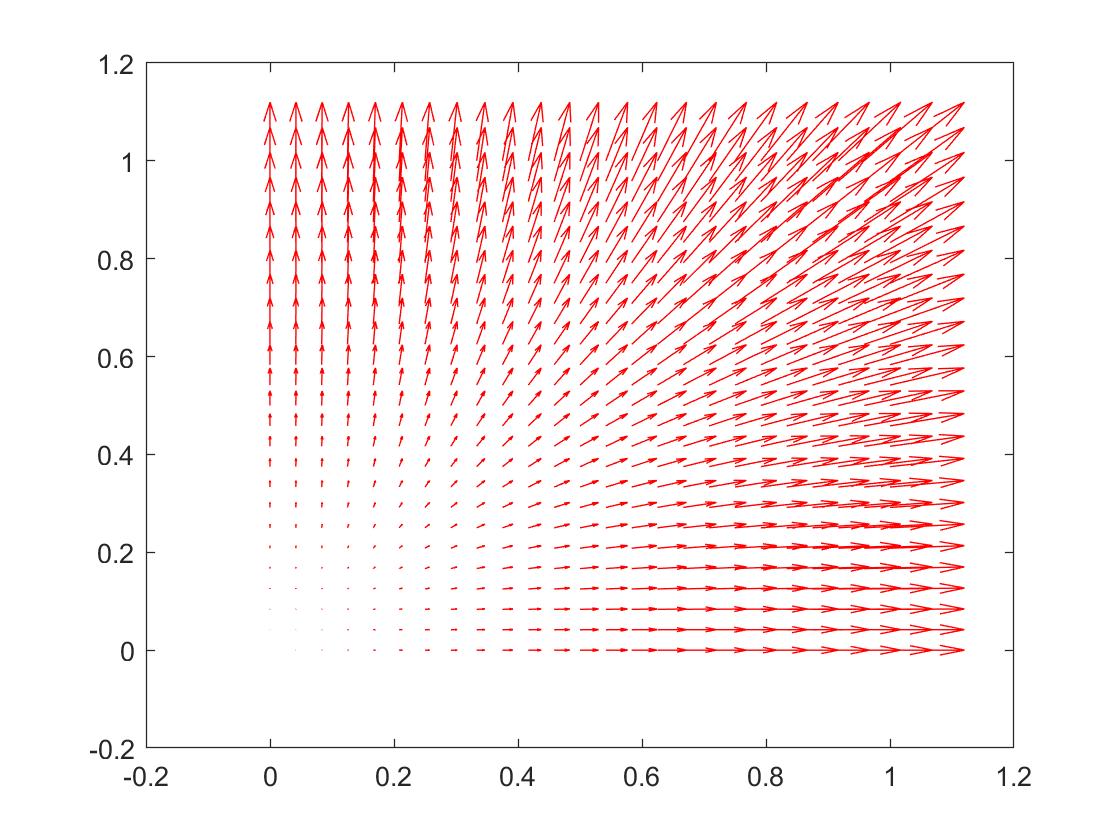}
	\caption{ Arrow plot of the computed displacement $ \bf u $ with the parameters of Table \ref{tab101}.}\label{figure_u_1}
\end{figure}	
\begin{table}[H]
	\begin{center}
		%\smallskip
		\caption{Values of parameters} \label{tab107}
		\begin{tabular}{l c c }
			\hline
			Parameters &Description  & Values  \\ \hline
			$\nu $ &  Poisson ratio   & 0.25      \\ %\hline
			$\alpha$ &  Biot-Willis constant   & 1e-5   \\ %\hline
			$E$ &  Young's modulus   & 2500 \\ %\hline
			$\lambda$ &  Lam$\acute{e}$ constant   & 1e3  \\%\hline
			$ K $ &  Permeability tensor  & (1e-3) $ \bf I$ \\% \hline
			$ \mu $ &  Lam$\acute{e}$ constant  & 1e3\\%\hline
			$ c_{0} $ &  Constrained specific storage coefficient  & 1\\\hline	
		\end{tabular}
	\end{center}
\end{table}
\begin{table}[H]
	\begin{center}
		\caption{Spatial errors and convergence rates of $\mathbf{u}$} \label{tab7}
		\begin{tabular}{l c c c c }
			\hline
			$h$ & $\|\mathbf{u}-\mathbf{u}_h\|_{L^2}$ &CR &$ \|\mathbf{u}-\mathbf{u}_h\|_{H^1}$ & CR\\ \hline
			$h=1/3$ & 1.2573e-6  &    & 1.8951e-5       &     \\ %\hline
			$h=1/6$ & 1.2922e-7  & 3.2824  & 3.2945-6     &  2.5241     \\ %\hline
			$h=1/12$ & 2.7945e-8  & 2.2092  & 6.9895e-7      & 2.2368      \\ %\hline
			$h=1/24$ & 3.2025-9  & 3.1253   & 1.9216e-7     &  1.8629      \\ \hline
		\end{tabular}
	\end{center}
\end{table}
\begin{table}[H]
	\begin{center}
		\caption{Spatial errors and convergence rates of $p$} \label{tab108}
		\begin{tabular}{l c c c c }
			\hline
			$h$ & $\|{p}-{p}_h\|_{L^2}$ &CR &$ \|{p}-{p}_h\|_{H^1}$ & CR\\ \hline
			$h=1/3$ & 0.024431 &    & 0.45261      &     \\ %\hline
			$h=1/6$ & 0.0049599  & 2.3003  & 0.22246    &  1.0247     \\ %\hline
			$h=1/12$ & 0.0010675  & 2.2161  & 0.10974      &  1.0195     \\ %\hline
			$h=1/24$ & 0.00024727  & 2.1101   & 0.054506     &  1.0096      \\ \hline
		\end{tabular}
	\end{center}
\end{table}	
\begin{figure}[H]
	\centering
	\includegraphics[height=2.2in,width=3in]{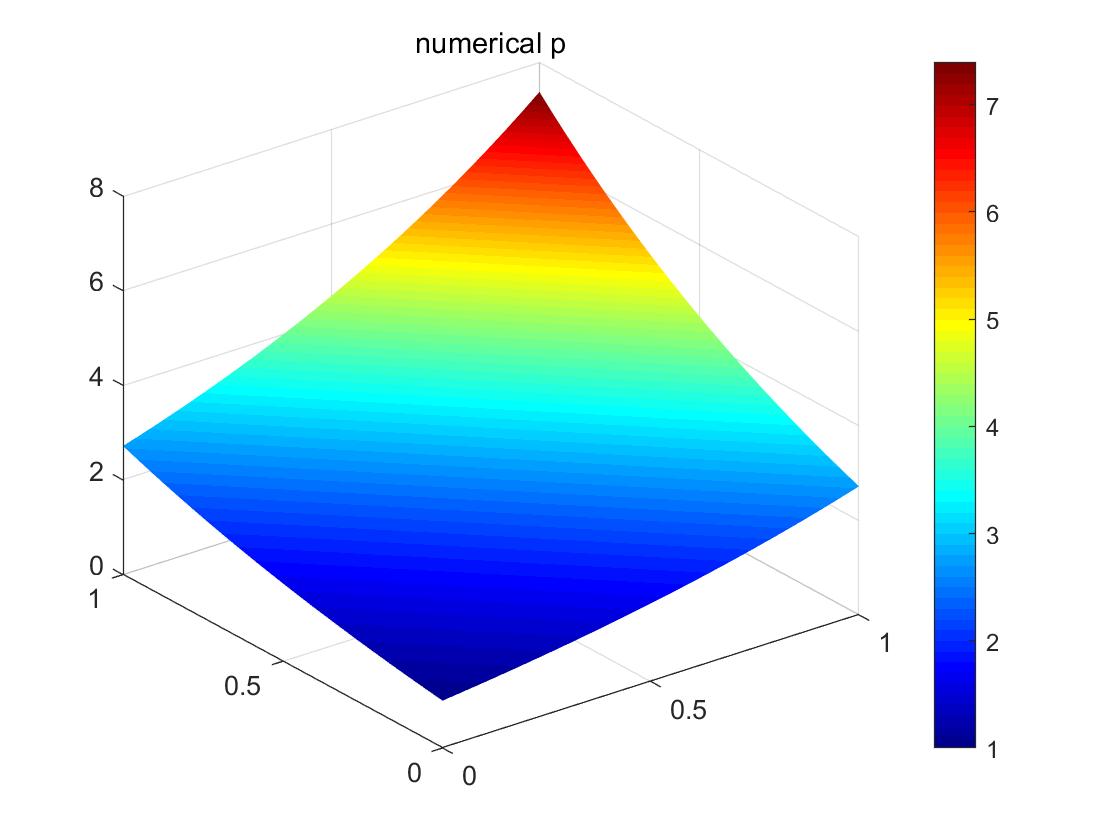}
	\caption{The numerical  pressure\ $p_{h}^{n+1}$ at the terminal time $T$ with the parameters of Table \ref{tab107}.}\label{figure_p5}
\end{figure}
\begin{figure}[H]
	\centering
	\includegraphics[height=2.2in,width=3in]{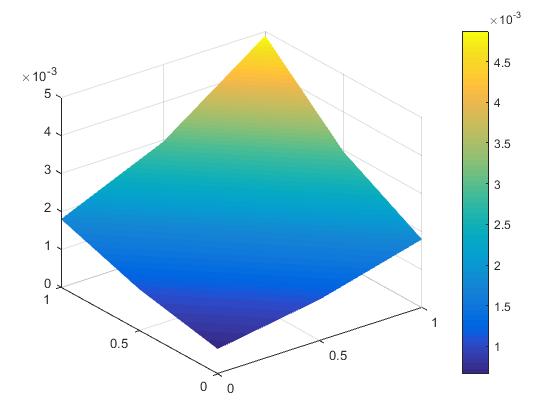}
	\caption{ Exact solution of pressure $p$ at the terminal time\ $T$ of Test 1.}\label{figure_p6}
\end{figure}
\begin{figure}[H]
	\centering
	\includegraphics[height=2.2in,width=3in]{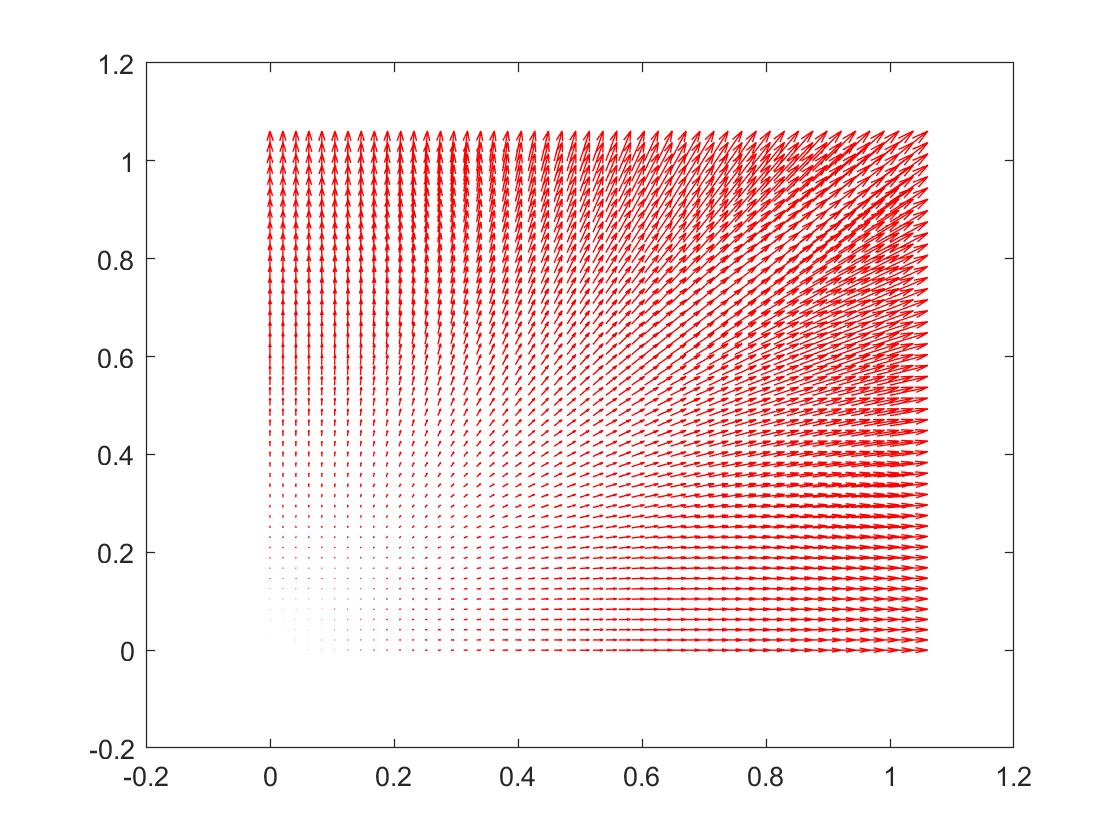}
	\caption{ Arrow plot of the computed displacement $ \bf u $ with the parameters of Table \ref{tab107}.}\label{figure_u_3}
\end{figure}
Table \ref{tab1} and Table \ref{tab106} display the error of displacement\ $\mathbf{u}$ and the pressure $p$ with $L^2(\Omega)$-norm and $H^1(\Omega)$-norm at the terminal time $T$ with the parameters of Table  \ref{tab101}, which are consistent with the theoretical result. Table \ref{tab7} and Table \ref{tab108} display the error of displacement\ $\mathbf{u}$ and the pressure $p$ with $L^2(\Omega)$-norm and $H^1(\Omega)$-norm at the terminal time $T$ with the parameters of Table \ref{tab107}. It is easy to find that there is  no ``locking phenomenon". 

Figure  \ref{figure_p1} and Figure \ref{figure_p5} show the numerical solution of pressure $p_{h}^{n+1}$ at the terminal time $T$ due to the difference of parameters between Table \ref{tab101} and Table \ref{tab107}, Figure \ref{figure_p6} shows the analytical solution of pressure $p_{h}^{n+1}$ at the terminal time $T$. Figure  \ref{figure_u_1} and Figure \ref{figure_u_3} show the arrow plot of the computed displacement $ \bf u $ corresponding to the parameters of Table \ref{tab101} and Table \ref{tab107}, respectively.

\medskip
{\bf Test 2.} Let $\Omega=[0,1]\times [0,1]$, $\Gamma_1=\{(1,x_2); 0\leq x_2\leq 1\}$,
$\Gamma_2=\{(x_1,0); 0\leq x_1\leq 1\}$,  $\Gamma_3=\{(0,x_2); 0\leq x_2\leq 1\}$,
$\Gamma_4=\{(x_1,1); 0\leq x_1\leq 1\}$, and $T=1$. The source functions are as follows:
\begin{align*}
	\mathbf{f} &=-(\lambda+\mu) t^2(1,1)^T-2(\mu+\lambda) t^{4}(x_{1},x_{2})^{T}+\alpha \cos(x_1+x_2)e^t(1,1)^T,\\
	\phi &=\Bigl(c_0+\frac{2K}{\mu_f} \Bigr)\sin(x_1+x_2)e^t+2t\alpha(x_1+x_2).
\end{align*}
and the boundary and initial conditions are as follows:
\begin{alignat*}{2}
	p &= \sin(x_1+x_2)e^t  &&\qquad\mbox{on }\partial\Omega_T,\\
	u_1 &= \frac12 x_1^2t^2 &&\qquad\mbox{on }\Gamma_j\times (0,T),\, j=1,3,\\
	u_2 &= \frac12 x_2^2t^2 &&\qquad\mbox{on }\Gamma_j\times (0,T),\, j=,2,4,\\
	\sigma\bf{n}-\alpha \emph{p}\bf{n} &= \mathbf{f}_1, &&\qquad \mbox{on } \p\Ome_T,\\
	\mathbf{u}(x,0) = \mathbf{0},  \quad p(x,0) &=\sin(x_1+x_2) &&\qquad\mbox{in } \Ome.
\end{alignat*}
where
\begin{align*}
	\mathbf{f}_1(x,t)= \lambda(x_1+x_2) (n_1,n_2)^T t^2 +\mu t^2(x_{1}n_{1},x_{2}n_{2})^{T}+\mu t^{4}\bigl(x_{1}^{2}n_{1},x_{2}^{2}n_{2}\bigr)^{T}\\
	+\lambda t^{4}(x_{1}^{2}+x_{2}^{2})(n_1,n_2)^T-\alpha\sin(x_1+x_2)(n_1,n_2)^T e^t.
\end{align*}
The exact solution of this problem is
\[
\mathbf{u}(x,t)=\frac{t^2}2 \bigl( x_1^2, x_2^2 \bigr)^T,\qquad p(x,t)=\sin(x_1+x_2)e^t.
\]
\begin{table}[!htbp]
	\begin{center}
		%\smallskip
		\caption{Values of parameters} \label{tab103}
		\begin{tabular}{l c c }
			\hline
			Parameters &Description  &Values  \\ \hline
			$\nu $ &  Poisson ratio   & 0.00495      \\ %\hline
			$\alpha$ &  Biot-Willis constant   & 1e-4   \\ %\hline
			$E$ &  Young's modulus   & 20.099 \\ %\hline
			$\lambda$ &  Lam$\acute{e}$ constant   & 0.1  \\%\hline
			$ K $ &  Permeability tensor  & 0.1 {\bf I} \\% \hline
			$ \mu $ & Lam$\acute{e}$ constant  & 10\\%\hline
			$ c_{0} $ &  Constrained specific storage coefficient  & 20\\ \hline	
		\end{tabular}
	\end{center}
\end{table}
\begin{table}[!htbp]
	\begin{center}
		\caption{Spatial errors and convergence rates of $\mathbf{u}$} \label{tab104}
	\begin{tabular}{l c c c c }
		\hline
		$h$ & $\|\mathbf{u}-\mathbf{u}_h\|_{L^2}$ &CR &$ \|\mathbf{u}-\mathbf{u}_h\|_{H^1}$ & CR\\ \hline
		$h=1/3$ & 1.1698e-7  &    & 4.0858e-7       &     \\ %\hline
		$h=1/6$ & 3.0566e-8  & 1.9363  & 1.0727e-7     &  1.9294     \\ %\hline
		$h=1/12$ & 7.7309e-9  & 1.9832  & 2.7139e-8      &  1.9828     \\ %\hline
		$h=1/24$ & 1.9451e-9  & 1.9908   & 6.8316e-9     &  1.9901      \\ \hline
		\end{tabular}
	\end{center}
\end{table}
\begin{table}[!htbp]
	\begin{center}
		\caption{Spatial errors and convergence rates of $p$} \label{tab105}
		\begin{tabular}{l c c c c }
			\hline
			$h$ & $\|p-p_h\|_{L^2}$ &CR &$ \|p-p_h\|_{H^1}$ & CR\\ \hline
			$h=1/3$ & 0.021791  &    & 0.29681      &     \\ %\hline
			$h=1/6$ & 0.005449  & 1.9997  & 0.14871     &  0.9970     \\ %\hline
			$h=1/12$ & 0.001368  & 1.9939  & 0.07439      & 0.9993      \\ %\hline
			$h=1/24$ & 0.000349  & 1.9708   & 0.03720     & 0.9998       \\ \hline
		\end{tabular}
	\end{center}
\end{table}
\begin{figure}[!htbp]
	\centering
	\includegraphics[height=2.2in,width=3in]{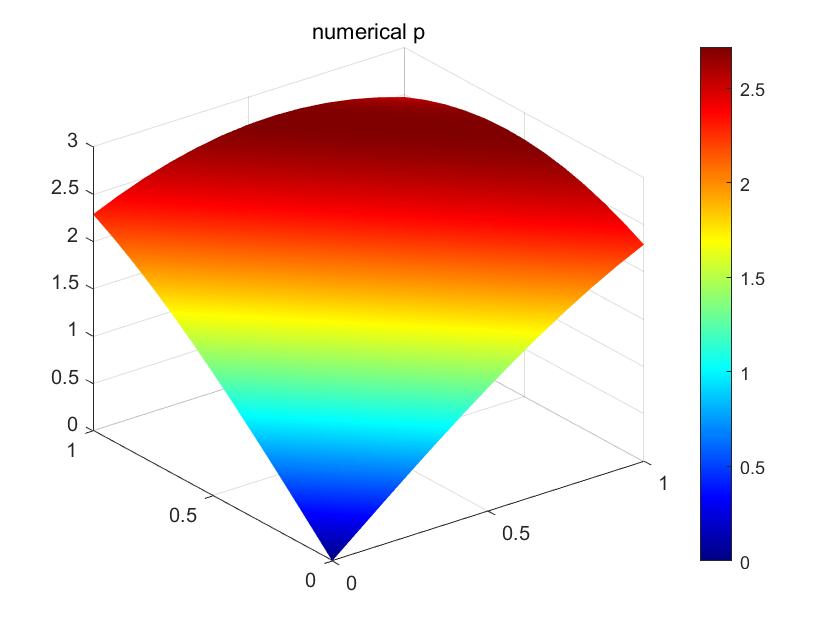}
	\caption{The numerical  pressure\ $p_{h}^{n+1}$ at the terminal time $T$ with the parameters of Table \ref{tab103}.}\label{figure_p3}
\end{figure}
\begin{figure}[!htbp]
	\centering
	\includegraphics[height=2.2in,width=3in]{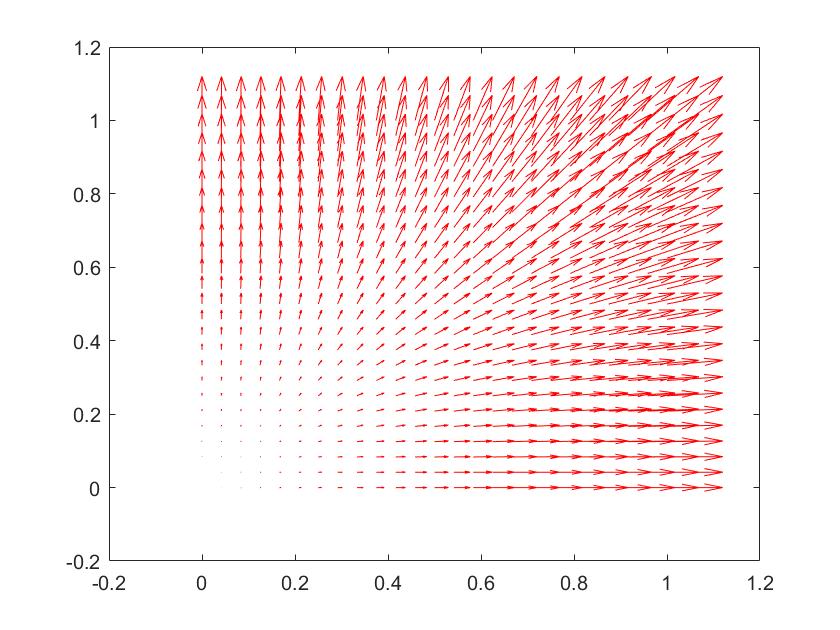}
	\caption{ Arrow plot of the computed displacement $ \bf u$ with the parameters of Table \ref{tab103}.}\label{figure_u_2}
\end{figure}
\begin{table}[!htbp]
	\begin{center}
		%\smallskip
		\caption{Values of parameters} \label{tab109}
		\begin{tabular}{l c c }
			\hline
			Parameters & Description  & Values  \\ \hline
			$\nu $ &  Poisson~ ratio   & 0.25      \\ %\hline
			$\alpha$ &  Biot-Willis~ constant   & 1e-4   \\ %\hline
			$E$ &  Young's~ modulus   & 2500 \\ %\hline
			$\lambda$ &  Lam$\acute{e}$ constant   & 1e3  \\%\hline
			$ K $ &  Permeability tensor  & (0.1) $ \bf I$ \\% \hline
			$ \mu $ &  Lam$\acute{e}$ constant  & 1e3\\%\hline
			$ c_{0} $ &  Constrained specific storage coefficient  & 0.01\\\hline	
		\end{tabular}
	\end{center}
\end{table}
\begin{table}[!htbp]
	\begin{center}
		\caption{Spatial errors and convergence rates of $\mathbf{u}$} \label{tab1010}
		\begin{tabular}{l c c c c }
			\hline
			$h$ & $\|\mathbf{u}-\mathbf{u}_h\|_{L^2}$ &CR &$ \|\mathbf{u}-\mathbf{u}_h\|_{H^1}$ & CR\\ \hline
			$h=1/3$ & 1.2573e-6  &    & 1.8951e-5       &     \\ %\hline
			$h=1/6$ & 1.2917e-7  & 3.2830  & 3.2945e-7     &  2.5241     \\ %\hline
			$h=1/12$ & 2.7942e-8  & 2.2088  & 6.9895e-7      &  2.2368     \\ %\hline
			$h=1/24$ & 3.2000e-9  & 3.1263   & 1.9216e-7     &  1.8629      \\ \hline
		\end{tabular}
	\end{center}
\end{table}
\begin{table}[!htbp]
	\begin{center}
		\caption{Spatial errors and convergence rates of $p$} \label{tab1011}
		\begin{tabular}{l c c c c }
			\hline
			$h$ & $\|p-p_h\|_{L^2}$ &CR &$ \|p-p_h\|_{H^1}$ & CR\\ \hline
			$h=1/3$ & 0.0218  &    & 0.2968      &     \\ %\hline
			$h=1/6$ & 0.0054  & 2.0133  & 0.1487     &  0.9971     \\ %\hline
			$h=1/12$ & 0.0014  & 1.9475  & 0.0744      & 0.9990      \\ %\hline
			$h=1/24$ & 0.00033949  & 2.0440   & 0.0372     & 1       \\ \hline
		\end{tabular}
	\end{center}
\end{table}
\begin{figure}[H]
	\centering
	\includegraphics[height=2.2in,width=3in]{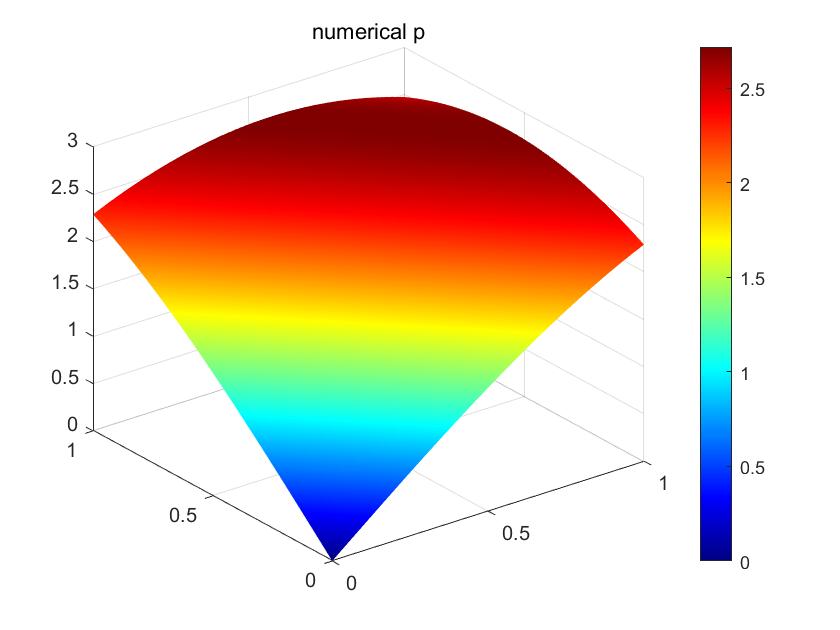}
	\caption{The numerical  pressure\ $p_{h}^{n+1}$ at the terminal time $T$ with the parameters of Table \ref{tab109}.}\label{figure_p7}
\end{figure}
\begin{figure}[H]
	\centering
	\includegraphics[height=2.2in,width=3in]{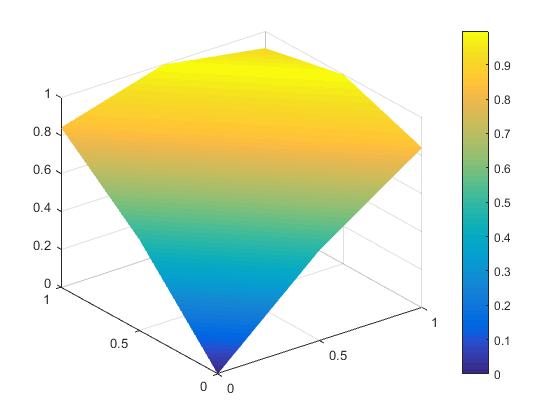}
	\caption{Exact solution of pressure $p$ at the terminal time\ $T$ of Test 2.}\label{figure_p8}
\end{figure}
\begin{figure}[H]
	\centering
	\includegraphics[height=2.2in,width=3in]{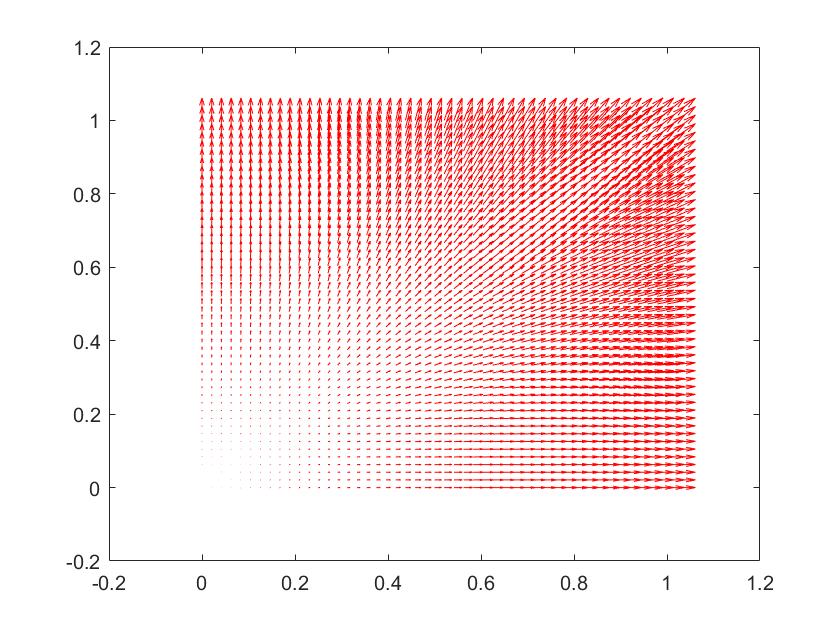}
	\caption{ Arrow plot of the computed displacement $ \bf u$ with the parameters of Table \ref{tab109}.}\label{figure_u_4}
	\end{figure}

Table \ref{tab104} and Table \ref{tab105} display the error of displacement $\mathbf{u}$ and the pressure $p$ with $L^2(\Omega)$-norm and $H^1(\Omega)$-norm at the terminal time $T$ with the parameters of Table \ref{tab103}, which are consistent with the theoretical result. Table \ref{tab1010} and Table \ref{tab1011} display the error of displacement $\mathbf{u}$ and the pressure $p$ with $L^2(\Omega)$-norm and $H^1(\Omega)$-norm at the terminal time $T$ with the parameters of Table \ref{tab109}, which show that our method overcomes the ``locking phenomenon".

Figure \ref{figure_p3} and Figure \ref{figure_p7} show the numerical solution of pressure $p_{h}^{n+1}$ at the terminal time $T$ due to the difference of parameters between Table \ref{tab103} and Table \ref{tab109}, Figure \ref{figure_p8} shows the analytical solution of pressure $p_{h}^{n+1}$ at the terminal time $T$. Figure  \ref{figure_u_2} and Figure \ref{figure_u_4} show the arrow plot of the computed displacement $ \bf u$ corresponding to the parameters of Table \ref{tab103} and Table \ref{tab109}, respectively.

\section{Conclusion}\label{sec-5}
In this paper, we propose a multiphysics finite element method and analyze the optimal error convergence order for a nonlinear poroelasticity model. Firstly, we reformulate the nonlinear fluid-solid coupling problem into a fluid-fluid coupling problem by a multiphysics approach. Secondly, we design a fully discrete time-stepping scheme to use multiphysics finite element method with $P_2-P_1-P_1$ element pairs  for the space variables and backward Euler method for the time variable, and we adopt the Newton iterative method to deal with the nonlinear term. Also, we derive the discrete energy laws and the optimal convergence order error estimates without any assumption on the nonlinear stress-strain relation. Finally, we show some numerical examples to verify the rationality of theoretical analysis and there is no ``locking phenomenon". To the best of our knowledge, the proposed fully discrete multiphysics finite element method for the nonlinear poroelasticity model is completely new.

%%%%

%\bigskip
%{\bf Acknowledgment.} The authors would like to thank two anonymous referees for their
%valuable questions and critical comments which helped to improve our paper, especially,
%to avoid an error in \eqref{e2.15e} and a subsequent error in the proof of Theorem \ref{thm2.6}.

%\vfill\eject
%%%%%%%%%%%%%%%%%%%%%%%%%%%%


\begin{thebibliography}{99}
	
	\bibitem{ber} M. Bercovier, O. Pironneau, {\em
		Error estimates for finite element solution of the Stokes problem in the primitive variables}, Numerische Mathematik, 1979, 33: 211-224.
	
	\bibitem{1} L. Berger, R. Bordas, D. Kay, et al, {\em A stabilized finite element method for finite-strain three-field poroelasticity}, Computational Mechanics, 2017, 60: 51-68.

	

	\bibitem{biot} M. Biot, {\em Theory of elasticity and consolidation for
	a porous anisotropic media}, Journal Applied Physics, 1955, 26: 182--185.

\bibitem{20210921} L. Bociu, G. Guidoboni, R. Sacco, J.T. Webster, {\em Analysis of nonlinear poro-elastic and poro-visco-elastic models}, Archive for Rational Mechanics and Analysis, 2016, 222(3): 1445-1519. 

\bibitem{20210922} L. Bociu, J.T. Webster, {\em Nonlinear quasi-static poroelasticity}, Journal of Differential Equations, 2021, 296: 242-278.


\bibitem{brenner} S. Brenner, {\em A nonconforming mixed multigrid method for the pure displacement problem in planar linear elasticity}, SIAM Journal Numerical Analysis, 1993, 30: 116-135.
	
\bibitem{bs08} S. Brenner, L.R. Scott, {\em The Mathematical Theory of Finite Element Methods},  third edition, Springer, 2008.
	
    \bibitem{brezzi} F. Brezzi, M. Fortin, {\em Mixed and
	Hybrid Finite Element Methods}, Springer, New York, 1992.

    
    \bibitem{20210923} Y. Cao, S. Chen, A.J. Meir, {\em Analysis and numerical approximations of equations of nonlinear poroelasticity}, Discrete and Continuous Dynamical Systems B, 2013, 18(5): 1253-1273.
    
    \bibitem{20210924} Y. Cao, S. Chen, A.J. Meir, {\em Quasilinear poroelasticity: analysis and hybrid finite element approximation}, Numerical Methods for Partial Differential Equations, 2015, 31(4): 1174-1189.

	\bibitem{11} D. Chapelle, J. Gerbeau, J. Sainte-Marie, I. Vignon-Clementel, {\em A poroelastic model
	valid in large strains with applications to perfusion in cardiac modeling}, Computational Mechanics, 2010, 46(1): 91¨C101.
	
    \bibitem{cia} P. Ciarlet, {\em The Finite Element Method for Elliptic
	Problems}, North-Holland, Amsterdam, 1978.

    \bibitem{coussy04} O. Coussy, {\em Poromechanics}, Wiley \& Sons, 2004.
    

    

   \bibitem{de86}
    M. Doi, S. Edwards, {\em The Theory of Polymer Dynamics},
    Clarendon Press, Oxford, 1986.
    
    \bibitem{20210927} C. Duijn, A. Mikelic, {\em Mathematical Theory of Nonlinear Single Phase Poroelasticity}, 2019, Preprint hal-02144933.

    \bibitem{20210820} L. Evans, {\em Partial Differential Equations}, American Mathematical Society, 2016.

	
	
	\bibitem{fglarxiv} X. Feng, Z. Ge, Y. Li, {\em Multiphysics finite element methods for a poroelasticity model}, arXiv:1411.7464, [math.NA], 2014.
	
	\bibitem{fgl14} X. Feng, Z. Ge, Y. Li, {\em Analysis of a multiphysics finite element method for a poroelasticity model}, IMA Journal of Numerical Analysis, 2018, 38: 330-359.
	
	\bibitem{fh10} X. Feng, Y. He, {\em Fully discrete finite element approximations of a polymer gel model}, SIAM Journal on Numerical Analysis, 2010, 48: 2186-2217.
	
	\bibitem{7} M. Ferronato, N. Castelletto, G. Gambolati, {\em A fully coupled 3-D mixed finite element model of Biot consolidation}, Computational physics, 2010, 229(12): 4813¨C4830.
	
	\bibitem{8} D. Gawin, P. Baggio, B. Schrefler, {\em Coupled heat, water and gas flow in deformable porous media}, International Journal for Numerical Methods in Fluids, 1995, 20: 969-978.
	

	\bibitem{201912096} Z. Ge, W. He, {\em Well-posedness of weak solution for Nonlinear Poroelasticity Model}, Preprint and submitted, 2021. arXiv: 2112.12425v1. 
	
	
		\bibitem{gra} V. Girault, P. Raviart, {\em
		Finite Element Method for Navier-Stokes Equations: theory and algorithms},
	Springer Verlag, Berlin, Heidelberg, New York, 1981.
	
	
	\bibitem{hamley07} I. Hamley, {\em Introduction to Soft Matter}, John Wiley \& Sons, 2007.
	
	\bibitem{6} J. Hudson, O. Stephansson, J. Andersson, C. Tsang, L. Ling, {\em Coupled T¨CH¨CM issues related to radioactive waste repository design and performance}, International
    Journal of Rock Mechanics and Mining Sciences, 2001, 38: 143¨C161.

    \bibitem{10} D. Nemec, J. Levec, {\em Flow through packed bed reactors: 1. single-phase flow}, Chemical Engineering Science, 2005, 60: 6947¨C6957.
      
    \bibitem{20210925} S. Owczarek, {\em A Galerkin method for Biot consolidation model}, Mathematics and Mechanics of Solids, 2010, 15(1): 42-56.
   
	\bibitem{3} W. Pao, R. Lewis, I. Masters, {\em A fully coupled hydro-thermo-poro-mechanical
	model for black oil reservoir simulation}, International Journal for Numerical and Analytical Methods in Geomechanics, 2001, 25: 1229¨C1256.

    \bibitem{pw07} P. Phillips, M. Wheeler, {\em A coupling of mixed
	and continuous Galerkin finite element methods for poroelasticity I:
	the continuous in time case}, Computational Geosciences, 2007, 11: 131-144.

    \bibitem{pw07b} P. Phillips, M. Wheeler, {\em A coupling of mixed
	and continuous Galerkin finite element methods for poroelasticity II:
	the discrete in time case}, Computational  Geosciences, 2007, 11: 145-158.

\bibitem{Shen1990} J. Shen, {\em Long time stability and convergence for fully discrete nonlinear Galerkin methods}, Appl. Anal., 1990, 38: 201¨C229.

    \bibitem{20210819} R. Showalter, {\em Diffusion in poro-elastic media}, Journal of Mathematical Analysis and Applications, 2000, 251: 310-340.

   \bibitem{temam}  R. Temam, {\em Navier-Stokes Equations}, Studies in Mathematics and its Applications, Vol. 2, North-Holland, 1977.

	\bibitem{2} A. Vuong, L. Yoshihara, W. Wall, {\em A general approach for modeling interacting flow through porous media under finite deformations}, Computer Methods in Applied Mechanics and Engineering, 2015, 283: 1240¨C1259.

    

    \bibitem{201912092} P. Wriggers, {\em Nonlinear Finite Element Methods}, Springer Verlag, 2008.

     \bibitem{20210928} A. 0—5en¨ª0"8ek, {\em The existence and uniqueness theorem in Biot's consolidation theory}, Aplikace Matematiky, 1984, 29(3): 194-211.
     
     
     
\end{thebibliography}
\end{document}